\documentclass{amsart}

\usepackage[latin2]{inputenc}
\usepackage{amsmath}
\usepackage{graphicx}
\usepackage{amssymb}
\usepackage{filecontents}
\usepackage{color}
\usepackage{esint}
\usepackage{epsfig}
\usepackage[english]{babel}

\newtheorem{theorem}{Theorem}
\newtheorem{proposition}[theorem]{Proposition}
\newtheorem{lemma}[theorem]{Lemma}

\newtheorem{definition}[theorem]{Definition}
\newtheorem{remark}[theorem]{Remark}

\def\Xint#1{\mathchoice
{\XXint\displaystyle\textstyle{#1}}%
{\XXint\textstyle\scriptstyle{#1}}%
{\XXint\scriptstyle\scriptscriptstyle{#1}}%
{\XXint\scriptscriptstyle\scriptscriptstyle{#1}}%
\!\int}
\def\XXint#1#2#3{{\setbox0=\hbox{$#1{#2#3}{\int}$ }
\vcenter{\hbox{$#2#3$ }}\kern-.6\wd0}}

\def\dashint{\Xint-}

\begin{document}

\allowdisplaybreaks[2]

\title[Entropy-Dissipating Reaction-Diffusion Equations]{Weak-Strong Uniqueness of Solutions to Entropy-Dissipating Reaction-Diffusion Equations}

\author{Julian Fischer}

\address{Institute of Science and Technology Austria (IST Austria), Am Campus 1, 3400 Klosterneuburg, Austria, E-Mail: julian.fischer@ist.ac.at}

\begin{abstract}
We establish a weak-strong uniqueness principle for solutions to entropy-dissipating reaction-diffusion equations: As long as a strong solution to the reaction-diffusion equation exists, any weak solution and even any renormalized solution must coincide with this strong solution. Our assumptions on the reaction rates are just the entropy condition and local Lipschitz continuity; in particular, we do not impose any growth restrictions on the reaction rates. Therefore, our result applies to any single reversible reaction with mass-action kinetics as well as to systems of reversible reactions with mass-action kinetics satisfying the detailed balance condition. Renormalized solutions are known to exist globally in time for reaction-diffusion equations with entropy-dissipating reaction rates; in contrast, the global-in-time existence of weak solutions is in general still an open problem -- even for smooth data -- , thereby motivating the study of renormalized solutions. The key ingredient of our result is a careful adjustment of the usual relative entropy functional, whose evolution cannot be controlled properly for weak solutions or renormalized solutions.
\end{abstract}

\maketitle

\section{Introduction}

Consider a general reversible chemical reaction of the form
\begin{align}
\label{MassActionEquation}
\alpha_1 \mathcal{A}_1+\ldots+\alpha_S \mathcal{A}_S
\rightleftharpoons
\beta_1 \mathcal{A}_1+\ldots+\beta_S \mathcal{A}_S,
\end{align}
where the $\mathcal{A}_i$ denote the different types of molecules and where $\alpha_i$, $\beta_i$ are nonnegative integers that denote the number of involved molecules of type $\mathcal{A}_i$.
An important model for the reaction kinetics of such reactions are \emph{mass action kinetics}: In mass action kinetics, the reaction rate is taken to be proportional to the probability that the involved reactants are simultaneously present in an infinitesimally small volume. Denoting the (nonnegative) concentration of the chemical $\mathcal{A}_i$ by $u_i$, in the example \eqref{MassActionEquation} the rate $R_f(u)$ of the forward reaction and the rate $R_b(u)$ of the backward reaction are therefore given by
\begin{align*}
R_f(u):= c_f\prod_{k=1}^S u_k^{\alpha_k},
\quad\quad
R_b(u):= c_b\prod_{k=1}^S u_k^{\beta_k},
\end{align*}
where $c_f,c_b>0$ are constants. The net rate of change of the concentration $u_i$ of the molecules of type $\mathcal{A}_i$ that is caused by the reaction is therefore
\begin{align}
\label{MassActionLaw}
R_i(u)=(\beta_i-\alpha_i)\left(c_f\prod_{k=1}^S u_k^{\alpha_k}
-c_b\prod_{k=1}^S u_k^{\beta_k}\right).
\end{align}
%The simplest corresponding reaction-diffusion equation reads
%\begin{align}
%\label{SimpleReactDiff}
%\frac{d}{dt} u_i = a_i\Delta u_i
%+(\beta_i-\alpha_i)\left(c_1\prod_{k=1}^S u_k^{\alpha_k}
%-c_2\prod_{k=1}^S u_k^{\beta_k}\right)
%~~~  \forall i\in \{1,\ldots,S\},
%\end{align}
%where the $a_i>0$ denote the (species-dependent) diffusion constants.
The mathematical analysis of reaction-diffusion equations with mass action kinetics poses interesting mathematical challenges: Even for the simple reaction-diffusion equation
\begin{align}
\label{SimpleReactDiff}
~~~~~~\frac{d}{dt} u_i = a_i\Delta u_i
+R_i(u) &&  \forall i\in \{1,\ldots,S\}
\end{align}
(with $a_i>0$ denoting species-dependent diffusion constants), until recently all proofs for the global existence of any kind of solution were limited to special cases \cite{BothePierre,CanizoDesvillettesFellner,CaputoVasseur,DesvillettesFellner,
DesvillettesEtAl,FiebachGlitzkyLinke,GajewskiGroeger,GajewskiSkrypnik,
GoudonVasseur,Kraeutle,MinchevaSiegel,PierreOverview}. Only recently, a general existence theory for reaction-diffusion equations with mass-action kinetics of the form
\begin{align}
\label{Equation}
\frac{d}{dt} u_i = \nabla \cdot (A_i \nabla u_i) - \nabla \cdot (u_i \vec b_i) +
R_i(u) &&  \forall i\in \{1,\ldots,S\}
\end{align}
(with general diffusion tensors $A_i$ and advection velocities $\vec b_i$ which may depend on space and time) has been developed by the author \cite{FischerReactDiffExistence}. The global solutions constructed in \cite{FischerReactDiffExistence} are so-called \emph{renormalized solutions}; whether weak solutions or even smooth solutions exist globally in time has remained an open problem.

The key difficulty in the proof of existence of solutions is the lack of control of the reaction terms. Although the global existence of smooth solutions is conjectured for simple reaction-diffusion equations like \eqref{SimpleReactDiff} with a single reversible reaction with mass-action kinetics \eqref{MassActionLaw}, there are no estimates available that would provide even just an $L^1$ bound for the reaction terms, even for smooth initial data: Besides the entropy dissipation estimate \eqref{EntropyDissipate} below, in general the only known bound is basically an $L^2(\Omega\times [0,T])$ estimate based on duality methods \cite{CanizoDesvillettesFellner,DesvillettesFellner,
DesvillettesEtAl,PierreSchmidt1,PierreSchmidt2}. Thus, for reaction rates with superquadratic growth there is not even a guarantee that the reaction terms $R_i(u)$ define a distribution, thereby obstructing any proof of (global-in-time) existence of weak solutions.
%In the literature, there have been various existence results for special cases of the equation \eqref{Equation} and related equations, assuming either growth restrictions on the reaction rates $R_i$ or (almost) equality of the diffusion coefficients for the different species; see e.\,g.\,  \cite{BothePierre,CanizoDesvillettesFellner,CaputoVasseur,FiebachGlitzkyLinke,
%GajewskiGroeger,GajewskiSkrypnik,GlitzkyHuenlich,GoudonVasseur,Kraeutle,MinchevaSiegel} 
%and the previous references.
Note that if one had an $L^1$ a priori bound for the reaction terms, the construction of weak solutions would be possible \cite{PierreNew,PierreL1}.

The most important mathematical energy estimate for reaction-diffusion equations with mass-action kinetics -- and, as discussed above, also almost the only energy estimate available -- is the \emph{entropy estimate}, which is a consequence of the structure \eqref{MassActionLaw} of the reaction-rates: There exist real numbers $\mu_i$ for which the entropy functional
\begin{align}
\label{Entropy}
E[u]:=\int_\Omega \sum_{i=1}^S u_i (\log u_i + \mu_i -1) \,dx
\end{align}
is dissipated along (sufficiently regular) solutions to the reaction-diffusion equation \eqref{SimpleReactDiff}. More precisely, given for example no-flux boundary conditions on  $\partial\Omega$, one has the dissipation estimate
\begin{align}
\label{EntropyDissipate}
E[u](T) + \sum_{i=1}^S \int_0^T \int_\Omega 4 a_i |\nabla \sqrt{u_i}|^2 \,dx \,dt
\leq E[u_0].
\end{align}
In fact, the author's theorem of existence of renormalized solutions for reaction-diffusion equations of the form \eqref{Equation} is not restricted to reaction rates of mass-action kinetics type, but (besides local Lipschitz continuity of the reaction rates and the non-consumption of chemicals which are not present, see (A6) below) only requires the entropy condition
\begin{align}
\label{EntropyCondition}
\sum_{i=1}^S R_i(u) (\log u_i+\mu_i)\leq 0\quad\quad\text{for all }u\in (\mathbb{R}_0^+)^S
\end{align}
for some $\mu_i\in \mathbb{R}$, $1\leq i\leq S$. Note that the entropy condition for the reaction rates entails the entropy dissipation estimate \eqref{EntropyDissipate} for simple reaction-diffusion equations of the form \eqref{SimpleReactDiff} and a similar estimate in the more general case \eqref{Equation}.

The entropy dissipation property -- and hence, also the existence theory for renormalized solutions in \cite{FischerReactDiffExistence} and the results of the present paper -- is not restricted to the situation of a single reversible reaction with mass action kinetics, but holds as well for systems of $N_R$ reversible reactions of the form
\begin{align}
\label{SystemOfEquations}
\alpha_1^n \mathcal{A}_1+\ldots+\alpha_S^n \mathcal{A}_S
\rightleftharpoons
\beta_1^n \mathcal{A}_1+\ldots+\beta_S^n \mathcal{A}_S,
~~~~ 1\leq n\leq N_R,
\end{align}
with corresponding mass-action kinetics
\begin{align}
\label{MassActionSystem}
R_i(u):= \sum_{n=1}^{N_R} (\beta_i^n-\alpha_i^n) \left(c_{f}^n\prod_{j=1}^S u_j^{\alpha_j^n}-c_{b}^n\prod_{j=1}^S u_j^{\beta_j^n}\right),
\end{align}
given that the so-called condition of detailed balance is satisfied. In particular, the entropy dissipation property holds for $N_R\leq S$ reversible reactions with mass-action kinetics provided that the matrix $(\beta_i^n-\alpha_i^n)_{in}$ has full rank (see e.\,g.\ \cite{GlitzkyMielke,SchusterSchuster}). See \cite{FeinbergHorn} for a mathematical analysis of such systems of reactions with mass-action kinetics.

It is also worth mentioning that reaction-diffusion equations with mass-action kinetics do not only dissipate the entropy \eqref{Entropy}, but may (formally and sometimes rigorously) be regarded as gradient flows of the entropy functional \cite{LieroMielke,LieroMielkeSavare,MielkeGradientStructure}.

While the entropy structure prevents global-in-space blowup of solutions, it does not provide pointwise control of solutions: Indeed, for reaction rates satisfying a condition that is closely related to the entropy condition -- namely the condition of dissipation of mass -- solutions featuring blowup in the $L^\infty$ norm have been constructed by Pierre and Schmitt \cite{PierreSchmidt1,PierreSchmidt2}. An overview of existence results for reaction-diffusion equations with dissipation of mass or dissipation of entropy may be found in the survey by Pierre \cite{PierreOverview}.

The author's paper \cite{FischerReactDiffExistence} provides an answer to the question of global existence of solutions to entropy-dissipating reaction-diffusion equations of the form \eqref{Equation}; however, it does not address the question of uniqueness.
In the present work, we provide a partial answer to the question of uniqueness of solutions: In Theorem~\ref{Theorem}, we prove that the \emph{existence} of a \emph{strong} solution to an entropy-dissipating reaction-diffusion equation on a certain time interval entails that this strong solution is also the \emph{unique} renormalized solution, as long as it exists. In the literature, results of this type are typically being referred to as \emph{weak-strong uniqueness} theorems.

Before sketching the strategy for the derivation of the weak-strong uniqueness result, let us briefly comment on the concept of \emph{renormalized solutions}. Renormalized solutions have originally been introduced by DiPerna and Lions in a series of seminal works \cite{DiPernaLions1,DiPernaLions,DiPernaLions2} in the setting of the continuity equation
\begin{align}
\label{ContinuityEquation}
\frac{d}{dt} u = -\nabla \cdot (u \vec{b})
\end{align}
and in the setting of the Boltzmann equation; since then, have found numerous applications in the theory of PDEs, see e.\,g.\ 
\cite{Alexandre,AlexandreVillani,DalMasoEtAl,Murat,Villani} and the references therein. To motivate the definition of renormalized solutions, consider the continuity equation for a vector field $\vec{b}\in L^1(\Omega)$ with $\nabla \cdot \vec{b} \in L^1(\Omega)$ and initial data $u_0\in L^1(\Omega)$. It becomes apparent that in this setting one can in general not give a meaning to the term $\nabla \cdot (u\vec{b})$ in the weak formulation of the continuity equation: The product $u\vec{b}$ is a product of $L^1$ functions (as the continuity equation in general has no regularizing effect), which in general does not even define a distribution. This motivates the introduction of a more general concept of solutions than weak solutions, namely renormalized solutions. A renormalized solution $u$ to the continuity equation is defined by the requirement that for all smooth functions $\xi:\mathbb{R}\rightarrow\mathbb{R}$ with compactly supported derivative $\xi'$, the function $\xi(u)$ must satisfy the equation derived from \eqref{ContinuityEquation} by a formal application of the chain rule: In other words, for all such $\xi$ the function $u$ must satisfy
\begin{align*}
\frac{d}{dt} \xi(u) = - \nabla \cdot (\xi(u) \vec{b}) + (\xi(u)-u\xi'(u)) \nabla \cdot \vec{b}
\end{align*}
in a weak sense, which is an equation that can be given a meaning in the sense of distributions in the setting $\vec{b}\in L^1(\Omega)$, $\nabla \cdot \vec{b} \in L^1(\Omega)$, $u_0\in L^1(\Omega)$.

Correspondingly, in the author's recent work \cite{FischerReactDiffExistence} renormalized solutions to the reaction-diffusion-advection equation \eqref{Equation} are defined by the condition that for all functions $\xi:(\mathbb{R}_0^+)^S\rightarrow \mathbb{R}$ with compactly supported derivative $D \xi$, the function $\xi(u)$ must satisfy the equation derived from \eqref{Equation} by a formal application of the chain rule in a weak sense; see Definition~\ref{DefinitionSolution} below for details.

Let us now briefly explain the mathematical concept that is central to our derivation of the weak-strong uniqueness theorem, the concept of so-called relative entropies.
As one easily checks by differentiation, the entropy \eqref{Entropy} is a strictly convex functional of $u$. The \emph{relative entropy} $E[u|v]$ is a mathematical concept to measure the ``distance'' of $u$ to some reference data $v$. It is obtained by subtracting an affine functional in $u$ from the convex entropy functional $E[u]$ in such a way that the resulting functional is nonnegative and has its unique zero for $u=v$: By definition, one has
\begin{align*}
E[u|v]&:=E[u]-DE[v](u-v) -E[v]
\\&
=\int_\Omega \sum_{i=1}^S u_i (\log u_i + \mu_i -1) \,dx
- \int_\Omega \sum_{i=1}^S  (u_i-v_i) (\log v_i + \mu_i) \,dx
\\&~~~
- \int_\Omega \sum_{i=1}^S v_i (\log v_i + \mu_i -1) \,dx
\\&
=\int_\Omega \sum_{i=1}^S \Big( u_i (\log u_i+\mu_i -1) - u_i (\log v_i+\mu_i) + v_i \Big) \,dx.
\end{align*}
Note that unlike the entropy $E[u]$ (which is dissipated for solutions of \eqref{SimpleReactDiff} with Neumann boundary data), the relative entropy $E[u|v]$ is in general not a nonincreasing function of time.

The advantage of the concept of relative entropies -- as opposed to other methods of measuring the ``distance'' of a solution $u$ to some reference data $v$, like $L^p$ norms or Sobolev norms -- is that relative entropies are often better adapted to the equation in consideration. For example, to evaluate the time derivative of the relative entropy $\frac{d}{dt} E[u|v]$, by the definition $E[u|v]:=E[u]-DE[v](u-v)-E[v]$ one basically just needs to use the entropy dissipation property of $u$ to estimate $\frac{d}{dt}E[u]$ and to test the weak formulation of the equation for $u$ with the test function $DE[v]$ -- that is in our setting, to test the equation for $u_i$ with $\log v_i+\mu_i$ and take the sum in $i$. In fact, numerous weak-strong uniqueness results for partial differential equations rely on relative entropies, for example the weak-strong uniqueness results for the compressible Navier-Stokes equation and related systems \cite{FeireislWeakStrong,FeireislNovotnyWeakStrongNavierStokesFourier}.

However, a direct application of the relative entropy method does not provide a weak-strong uniqueness result for entropy-dissipating reaction-diffusion equations without substantial additional ideas, even when assuming arbitrary smoothness and positivity properties of the strong solution $v$: By a formal computation, we have for two solutions $u$ and $v$ of our equation \eqref{SimpleReactDiff} with no-flux boundary conditions
\begin{align}
\label{FormalUsualRelativeEntropy}
\frac{d}{dt} E[u|v]
=&-\int_\Omega \sum_{i=1}^S a_i u_i \bigg|\frac{\nabla u_i}{u_i}-\frac{\nabla v_i}{v_i}\bigg|^2 \,dx
\\&
\nonumber
+\int_\Omega \sum_{i=1}^S R_i(u) \Big(\log \frac{u_i}{v_i}+\mu_i-\mu_i\Big)
-\sum_{i=1}^S R_i(v) \left(\frac{u_i}{v_i}-1\right) \,dx.
\end{align}
Typically, one would now try to estimate the right-hand side from above in terms of the relative entropy and use a Gronwall-type argument to conclude that for a renormalized or weak solution $u$ and a strong solution $v$ to the equation \eqref{SimpleReactDiff} with the same initial data, one has $E[u|v]=0$ for all $T\geq 0$ and therefore $u=v$ almost everywhere.

The key obstacle to proving weak-strong uniqueness using the relative entropy $E[u|v]$ is the insufficient control of the term
\begin{align}
\label{CriticalTerm}
\int_\Omega \sum_{i=1}^S R_i(u) \Big(\log \frac{1}{v_i}-\mu_i\Big) \,dx
\end{align}
whenever $u$ is just a weak solution or a renormalized solution to the reaction-diffusion equation \eqref{SimpleReactDiff}, even when $v_i$ is assumed to be smooth and strictly positive and even if we are in the case of a single reaction with mass-action kinetics \eqref{MassActionLaw}. Any attempt of controlling the term \eqref{CriticalTerm} in terms of the nonpositive (see \eqref{EntropyCondition}) dissipation terms from \eqref{FormalUsualRelativeEntropy}
\begin{align*}
-\int_\Omega \sum_{i=1}^S a_i u_i \bigg|\frac{\nabla u_i}{u_i}-\frac{\nabla v_i}{v_i}\bigg|^2 \,dx+\int_\Omega \sum_{i=1}^S R_i(u) (\log u_i+\mu_i) \,dx
\end{align*}
or the relative entropy fails in general: For example, for reaction terms of the form \eqref{MassActionLaw} (for simplicity of the outline, let us take $c_f=c_b=1$ and $\mu_i=0$) the second dissipation term may be rewritten as
\begin{align*}
\int_\Omega \sum_{i=1}^S R_i(u)(\log u_i+\mu_i) \,dx = -\int_\Omega (R_f(u)-R_b(u)) \log \frac{R_f(u)}{R_b(u)} \,dx,
\end{align*}
which in general does not provide any control of the net reaction rate $R_i(u)=(\beta_i-\alpha_i)(R_f(u)-R_b(u))$. In fact, in case $\frac{R_f(u)}{R_b(u)}\approx 1$ elementary calculus shows that the dissipation behaves like
\begin{align*}
-\frac{|R_f(u)-R_b(u)|^2}{R_b(u)}.
\end{align*}
Therefore in such a case one cannot rule out that the term
\begin{align*}
\int_\Omega \sum_{i=1}^S R_i(u)\Big(\log \frac{1}{v_i}-\mu_i\Big) \,dx
\end{align*}
might strongly dominate the dissipation
\begin{align*}
\int_\Omega \sum_{i=1}^S R_i(u)(\log u_i+\mu_i) \,dx
\end{align*}
and lead to $E[u|v]$ becoming strictly positive. Note that the last term in \eqref{FormalUsualRelativeEntropy} $-\int_\Omega \sum_{i=1}^S R_i(v) \left(\frac{u_i}{v_i}-1\right) \,dx$  is better-behaved than the term \eqref{CriticalTerm}, thus there is no hope for cancellations. Due to the strong polynomial growth of $R_i(u)$ in $u$, the relative entropy $E[u|v]$ itself does not provide any control of the term \eqref{CriticalTerm} either.

The key idea of our result is to instead consider the \emph{adjusted} relative entropy functional
\begin{align}
E_M[u|v]:= \int_\Omega \sum_{i=1}^S \Big( u_i (\log u_i+\mu_i-1) - \xi_M(u) u_i (\log v_i+\mu_i) + v_i \Big) \,dx,
\end{align}
where $\xi_M:\mathbb{R}^S \rightarrow [0,1]$ is a cutoff with $\xi_M(w)=1$ for $\sum_{i=1}^S w_i \leq M$ and $\xi_M(w)=0$ for $\sum_{i=1}^S w_i \geq M^K$ as well as $|\partial_j \xi_M(w)|\leq \frac{C}{K \sum_{i=1}^S w_i}$, $|\partial_j\partial_k \xi_M(w)|\leq \frac{C}{K |\sum_{i=1}^S w_i|^2}$. Here, the cutoff concentration $M$ is chosen fixed but much larger than the maximum of the strong solution $v$; the constant $K\geq 2$ is chosen fixed but large enough depending on the strong solution $v$. For this adjusted entropy functional, formal computations show that for sufficiently regular and strictly positive strong solutions $v$, a Gronwall-type argument is now applicable and yields weak-strong uniqueness: The time derivative of the adjusted relative entropy is given by
\begin{align}
\label{FormalRelativeEntropy}
&\frac{d}{dt} E_M[u|v]
\\&
\nonumber
=\sum_{i=1}^S \int_\Omega - a_i u_i \bigg|\frac{\nabla u_i}{u_i}\bigg|^2 - a_i u_i \xi_M(u) \bigg|\frac{\nabla v_i}{v_i}\bigg|^2 + 2 a_i u_i \xi_M(u) \frac{\nabla u_i}{u_i}\cdot \frac{\nabla v_i}{v_i} \,dx
\\&~~~
\nonumber
+\sum_{i,k=1}^S \int_\Omega a_k \nabla u_k \cdot \nabla (\partial_k \xi_M(u) u_i (\log v_i+\mu_i)) \,dx
\\&~~~
\nonumber
+\sum_{i=1}^S \int_\Omega a_i \nabla u_i \cdot (\log v_i+\mu_i) \nabla (\xi_M(u))
+a_i \frac{u_i}{v_i} \nabla v_i \cdot \nabla (\xi_M(u))\,dx
\\&~~~
\nonumber
+ \sum_{i=1}^S \int_\Omega R_i(u)\big(\log u_i+\mu_i-\xi_M(u) (\log v_i+\mu_i)\big) - R_i(v)\Big(\xi_M(u)\frac{u_i}{v_i}-1 \Big) \,dx
\\&~~~
\nonumber
-\sum_{i,k=1}^S \int_\Omega \partial_k \xi_M(u) u_i (\log v_i+\mu_i) \, R_k(u) \,dx.
\end{align}
Now, to apply Gronwall, one pulls all integrals together and argues in a pointwise fashion: For $\sum_{i=1}^S u_i\leq M$, one has $\xi_M(u)=1$ and therefore one obtains the simplified integrand
\begin{align}
\label{DissipationEstimateSmallerthanM}
-\sum_{i=1}^S a_i u_i \left|\frac{\nabla u_i}{u_i}-\frac{\nabla v_i}{v_i}\right|^2
+\sum_{i=1}^S \left(R_i(u) \log \frac{u_i}{v_i}-R_i(v) \left(\frac{u_i}{v_i}-1\right)\right).
\end{align}
Note that due to $\sum_{i=1}^S u_i\leq M$ and $\sum_{i=1}^S v_i\leq M$, for the second sum an upper bound of the form $C(\sup_{x,t,i} v_i,\inf_{x,t,i} v_i,M,R_i) \sum_{i=1}^S |u_i-v_i|^2$ is available (where by $C(\ldots,R_i)$ we mean that the constant depends on the functions $R_i$). By uniform convexity of $w \log w$ on bounded sets, we have (and shall prove rigorously in Lemma~\ref{Lemma7}) an estimate of the form
\begin{align*}
\int_\Omega |u_i-v_i|^2 \chi_{\{\sum_{i=1}^S u_i\leq M\}} \,dx \leq C(M) E_M[u|v].
\end{align*}
The integral of the terms \eqref{DissipationEstimateSmallerthanM} over the set $\{\sum_{i=1}^S u_i\leq M\}$ is therefore bounded by a Gronwall-type term $C(\inf_{x,t,i} v_i,M,R_i) E_M[u|v]$.

In the case $\sum_{i=1}^S u_i\geq M^K$, one has $\xi_M(u)=0$. Thus, in this case the integrand becomes
\begin{align*}
-\sum_{i=1}^S a_i u_i \bigg|\frac{\nabla u_i}{u_i}\bigg|^2
+\sum_{i=1}^S R_i(u) (\log u_i+\mu_i) + \sum_{i=1}^S R_i(v),
\end{align*}
which (since in view of the entropy condition \eqref{EntropyCondition} only the last of the three sums might be nonnegative) for $M$ large enough is clearly bounded from above by $C(||v||_{L^\infty}) \sum_{i=1}^S \big(u_i(\log u_i+\mu_i-1)-\xi_M(u)u_i (\log v_i+\mu_i)+v_i\big)$.

The remaining case $M<\sum_{i=1}^S u_i<M^K$ is the crucial case. For $M$ and $K$ large enough (depending on $\sup_{i,x,t} v_i$) but fixed, the adjusted relative entropy density $\sum_{i=1}^S (u_i (\log u_i+\mu_i-1)-\xi_M(u)u_i(\log v_i+\mu_i)+v_i)$ is bounded from below by $1$ in this case (note that in particular $u=v$ is excluded by the choice of $M$; for details of this lower bound, see Lemma~\ref{Lemma7}).
Furthermore, all reaction terms are bounded by a constant $C(M,K,R_i,\inf_{x,t,i} v_i)$. Therefore, the integral of all reaction terms in the case $M<\sum_{i=1}^S u_i<M^K$ may be bounded by a Gronwall term $C(M,K,R_i,\inf_{x,t,i} v_i) E_M[u|v]$. It remains to consider the diffusion terms. Applying the chain rule, the terms in the first and the second line of the right-hand side of \eqref{FormalRelativeEntropy} read (we omit the discussion of the terms from the third line of the right-hand side, as they may be treated analogously) 
\begin{align*}
&-\sum_{i=1}^S a_i u_i \left|\frac{\nabla u_i}{u_i}\right|^2
-\sum_{i=1}^S a_i u_i \xi_M(u) \left|\frac{\nabla v_i}{v_i}\right|^2
+\sum_{i=1}^S 2a_i u_i \xi_M(u) \frac{\nabla u_i}{u_i}\cdot \frac{\nabla v_i}{v_i}
\\&
+\sum_{i,k=1}^S a_k u_k \frac{\nabla u_k}{u_k} \cdot \partial_k \xi_M(u) u_i (\log v_i+\mu_i) \frac{\nabla u_i}{u_i}
+\sum_{i,k=1}^S a_k u_k \frac{\nabla u_k}{u_k} \cdot \partial_k \xi_M(u) u_i \frac{\nabla v_i}{v_i}
\\&
+\sum_{i,k,l=1}^S a_k u_k \frac{\nabla u_k}{u_k} \cdot \partial_k \partial_l \xi_M(u) u_i u_l (\log v_i+\mu_i) \frac{\nabla u_l}{u_l}.
\end{align*}
By the lower bound on the entropy density in the case $M<\sum_{i=1}^S u_i<M^K$, one may estimate the integral of terms of the form $\sum_{i=1}^S u_i \big|\frac{\nabla v_i}{v_i}\big|^2$ (and similar terms) by a Gronwall term $C E_M[u|v]$ with a constant $C=C(M,K,\inf_{x,t,i}v_i,\sup_{x,t,i}|\nabla v_i|)$. It is therefore possible to estimate the last term in the first line and the last term in the second line by Young's inequality, absorbing the terms involving $u_i |\nabla u_i/u_i|^2$ in the first term in the first line and generating a remaining Gronwall term of the form $C(M,K,\inf_{x,t,i}v_i,\sup_{x,t,i}|\nabla v_i|) E_M[u|v]$. To estimate the remaining terms, namely the first term in the second line and the term in the third line, one makes the crucial observation that for $K$ large enough, these terms may be absorbed in the first term in the first line: The estimates $|\partial_k \xi_M(w)|\leq \frac{C}{K\sum_{i=1}^S w_i}$ and $|\partial_k \partial_l \xi_M(w)|\leq \frac{C}{K(\sum_{i=1}^S w_i)^2}$ facilitate this absorption, provided that $K$ is chosen large enough (depending on $\sup_{i,x,t} |\log v_i|$ and the data).

In conclusion, for $M$ and $K$ large enough we obtain an estimate of the form
\begin{align*}
\frac{d}{dt} E_M[u|v] \leq C\Big(M,K,R_i,\sup_{i,x,t}v_i(x,t),\inf_{i,x,t}v_i(x,t),\sup_{i,x,t}|\nabla v_i|\Big) E_M[u|v].
\end{align*}
By the Gronwall lemma, this estimate entails $E_M[u|v](T)=0$ for all $T$ if the initial data of $u$ and $v$ coincide (which implies $E_M[u|v](0)=0$).

{\bf Notation.}
We shall use the abbreviation $I:=[0,\infty)$ for the time interval $[0,\infty)$. By $L^p_{loc}(I;X)$ we denote the set of functions whose restrictions to $[0,T]$ belong to $L^p([0,T];X)$ for any $T<\infty$ ($X$ being an arbitrary Banach space). For a domain $\Omega\subset\mathbb{R}^d$, we denote by $H^1(\Omega)$ the space of functions $u\in L^2(\Omega)$ whose distributional derivative $\nabla u$ is square-integrable; on $H^1(\Omega)$ we use the standard norm $||u||_{H^1(\Omega)}^2 := \int_\Omega |u|^2 + |\nabla u|^2 \,dx$. As usual, by $\chi_{A}$ we denote the characteristic function of a set $A$. For a set $A$, $\partial A$ refers to its boundary. The space of smooth compactly supported functions on a domain $\Omega\subset\mathbb{R}^d$ is denoted by $C^\infty_{cpt}(\Omega)$; by $C^\infty(\overline{\Omega})$ we denote the set of functions on $\Omega$ that admit a smooth extension to $\mathbb{R}^d$. By $\vec n$ we denote the outer unit normal vector to a domain $\Omega\subset \mathbb{R}^d$. We denote the $(d-1)$-dimensional Hausdorff measure by $\mathcal{H}^{d-1}$.

\section{Main Results}

As in \cite{FischerReactDiffExistence}, our key assumption on the reaction rates is the condition of entropy dissipation (\ref{EntropyCondition}), i.\,e.\ the existence of $\mu_i\in \mathbb{R}$ for which the dissipation
\begin{align*}
\sum_{i=1}^S R_i(u) (\log u_i+\mu_i)\leq 0\quad\quad\text{for all }u\in (\mathbb{R}_0^+)^S
\end{align*}
holds.

Besides the entropy dissipation condition, we again impose the following (modest) conditions on our domain, the coefficients, and the reaction rates.
\begin{itemize}
  \item[(A1)] Let $d\in \mathbb{N}$ and let $\Omega \subset \mathbb{R}^d$ be a
  bounded Lipschitz domain.
  \item[(A2)] Assume that $A_i\in
  \left[L^\infty(I;L^\infty(\Omega))\right]^{d\times d}$.
  \item[(A3)] Suppose that there exists $\lambda>0$ such that for all $i$, 
  all $x\in \Omega$, all $t\in I$, and all $v\in \mathbb{R}^d$ we have
  $A_i(x,t) v \cdot v\geq  \lambda |v|^2$.
  \item[(A4)] Assume that $\vec b_i\in \left[L^\infty(I;
  L^\infty(\Omega))\right]^d$ and that the trace of $\vec b_i$ on the spatial
  boundary $\partial \Omega\times I$ exists.%; we then define
  % $\beta:=\max_i||\vec b_i||_{L^\infty}$.}
  \item[(A5)] Let $R_i:\left(\mathbb{R}_0^+\right)^{S}\rightarrow \mathbb{R}$
  be locally Lipschitz continuous for all $1\leq i\leq S$, that is Lipschitz continuous on every bounded subset of $(\mathbb{R}_0^+)^S$.
  \item[(A6)] Assume that $R_i(v)\geq 0$ holds for any $v\in (\mathbb{R}_0^+)^S$ with
  $v_i=0$.
\end{itemize}
Note that the condition (A6) guarantees that a chemical which is not present cannot be consumed by reactions, thereby ensuring nonnegativity of all concentrations $u_i$. By $R$ we shall denote the vector-valued function whose $S$ components are given by $R_1,\ldots,R_S$.

We impose the same boundary conditions as in \cite{FischerReactDiffExistence}:
\begin{itemize}
  \item[(B1)] Let $\Gamma_{In}$, $\Gamma_{Out}$ be disjoint open subsets of $\partial\Omega$ with $\overline{\Gamma}_{In}\cup\overline{\Gamma}_{Out}=\partial\Omega$.
  \item[(B2)] Let $g_i\in L^\infty(\Gamma_{In}\times I)$, $1\leq i\leq S$, be bounded nonnegative functions. On $\Gamma_{In}\times I$, we impose the boundary condition $\vec n \cdot(A_i \nabla u_i- u_i \vec b_i)=g_i$.
  \item[(B3)] Suppose that on $\Gamma_{Out}$ we have $\vec n\cdot \vec b_i\geq 0$.
  On $\Gamma_{Out}\times I$ we then impose the boundary condition $\vec n\cdot A_i
  \nabla u_i=0$.
\end{itemize}
Condition (B2) prescribes the (in-)flow through the boundary $\Gamma_{In}$. Condition (B3) requires the diffusive flux at the boundary to vanish and is a typical outflow boundary condition on $\Gamma_{Out}$.

Note that one could in principle also work with other classes of boundary conditions, as long as they admit a suitable entropy estimate: For example, it would be possible to establish the existence theory of \cite{FischerReactDiffExistence} (and also the results in the present paper) for boundary reactions -- that is, for the boundary condition $\vec{n}\cdot (A_i\nabla u_i-u_i \vec{b}_i)=R^{\partial\Omega}_i(u)$ -- as long as the boundary reaction rates $R^{\partial\Omega}_i(\cdot)$ are locally Lipschitz continuous, satisfy (A6), and are subject to the entropy condition $\sum_{i=1}^S R^{\partial\Omega}_i(w) (\log w_i+\mu_i)\leq 0$ for all $w\in (\mathbb{R}_0^+)^S$ for the same $\mu_i$ as in \eqref{EntropyCondition}.

\begin{definition}[Weak solutions]
\label{DefinitionWeakSolution}
Suppose that (A1)-(A6) hold.
Let $(u_0)_i\in L^1(\Omega)$, $1\leq i\leq S$, be nonnegative. Let $0<T_{max}\leq \infty$. We say that nonnegative functions $u_i\in L^\infty_{loc}([0,T_{max});L^1(\Omega))$ with $\sqrt{u_i}\in L^2_{loc}([0,T_{max});H^1(\Omega))$, $1\leq i\leq S$, are a weak solution to the reaction-diffusion-advection equation (\ref{Equation}) with initial data $u_0$ and boundary conditions (B1)-(B3) provided that we have
\begin{align*}
R_i(u)\in L^1(\Omega\times [0,T])
\end{align*}
for all $1\leq i\leq S$ and any $T<T_{max}$ and that for any $\psi\in C^\infty(\overline{\Omega}\times I)$ and any $1\leq i\leq S$ the equation
\begin{align}
\nonumber
&\int_\Omega u_i(\cdot,T) \psi(\cdot,T) ~dx
-\int_\Omega (u_0)_i \psi(\cdot,0) ~dx
-\int_0^T \int_\Omega u_i \frac{d}{dt}\psi ~dx ~dt
\\
\nonumber
=&-\int_0^T \int_\Omega (A_i\nabla u_i) \cdot \nabla \psi ~dx ~dt
\\&
\label{WeakFormulation}
+\int_0^T \int_\Omega u_i \vec b_i \cdot \nabla \psi ~dx ~dt
\\&
\nonumber
+\int_0^T \int_\Omega R_i(u) \psi ~dx~dt
\\&
\nonumber
+\int_0^T \int_{\Gamma_{In}} g_i \psi ~d\mathcal{H}^{d-1} ~dt
\\&
\nonumber
-\int_0^T \int_{\Gamma_{Out}} \vec n \cdot \vec b_i ~u_i ~\psi ~d\mathcal{H}^{d-1} ~dt
\end{align}
is satisfied for almost every $T<T_{max}$.
\end{definition}

Recall the following definition of renormalized solutions to the reaction-diffusion-advection equation \eqref{Equation} introduced in \cite{FischerReactDiffExistence}.
\begin{definition}[Renormalized solutions]
\label{DefinitionSolution}
Suppose that (A1)-(A6) hold.
Let $(u_0)_i\in L^1(\Omega)$, $1\leq i\leq S$, be nonnegative.
We say that nonnegative
functions $u_i\in L^\infty_{loc}(I;L^1(\Omega))$ with $\sqrt{u_i}\in
L^2_{loc}(I;H^1(\Omega))$, $1\leq i\leq S$, are a renormalized solution to the
reaction-diffusion-advection equation (\ref{Equation}) with initial data $u_0$
and boundary conditions (B1)-(B3) if for every smooth function
$\xi:(\mathbb{R}_0^+)^S\rightarrow \mathbb{R}$ with compactly supported
derivative $D\xi$ and for every $\psi\in C^\infty(\overline{\Omega}\times I)$
the equation
\begin{align}
\nonumber
&\int_\Omega \xi(u(\cdot,T)) \psi(\cdot,T) ~dx
-\int_\Omega \xi(u_0) \psi(\cdot,0) ~dx
-\int_0^T \int_\Omega \xi(u) \frac{d}{dt}\psi ~dx ~dt
\\
\nonumber
=
&-\sum_{i,j=1}^S \int_0^T \int_\Omega \psi \partial_i\partial_j \xi(u)
(A_i\nabla u_i) \cdot \nabla u_j ~dx ~dt
\\
\label{RenormalizedFormulation}
&-\sum_{i=1}^S \int_0^T \int_\Omega \partial_i \xi(u)
(A_i\nabla u_i) \cdot \nabla \psi ~dx ~dt
\\&
\nonumber
+\sum_{i,j=1}^S \int_0^T \int_\Omega \psi \partial_i\partial_j\xi(u)
u_i \vec b_i \cdot \nabla u_j ~dx ~dt
\\&\nonumber
+\sum_{i=1}^S \int_0^T \int_\Omega \partial_i \xi(u) u_i \vec b_i \cdot
\nabla \psi ~dx ~dt
\\&
\nonumber
+\sum_{i=1}^S \int_0^T \int_\Omega \partial_i \xi(u) R_i(u) \psi ~dx~dt
\\&
\nonumber
+\sum_{i=1}^S \int_0^T \int_{\Gamma_{In}} g_i \psi \partial_i \xi(u)
~d\mathcal{H}^{d-1} ~dt
\\&
\nonumber
-\sum_{i=1}^S \int_0^T \int_{\Gamma_{Out}} \vec n \cdot \vec b_i ~u_i ~\psi
\partial_i \xi(u) ~d\mathcal{H}^{d-1} ~dt
\end{align}
is satisfied for almost every $T>0$.
\end{definition}
The crucial property of the definition of renormalized solutions is that it makes sense even if (integrable) singularities in the functions $u$ cause the reaction rates $R_i(u)$ to become nonintegrable, i.\,e.\ even for $R_i(u)\notin L^1(\Omega \times [0,T])$, a case that cannot be handled by any concept of weak solutions: Note that we have not imposed any growth restrictions on the reaction rates $R_i$; therefore, in principle even a very mild failure of $v\in L^\infty(\Omega \times [0,T])$ might cause $R_i(v)$ to become nonintegrable, at least for generic functions $v$. Whether there exist  reaction rates $R_i$ subject to our assumptions and renormalized solutions $u$ for which $R_i(u)$ is actually nonintegrable is an open question; however, at the moment, there is no technique available that would exclude the occurrence of such a situation either, even for the case of a single reaction with mass-action kinetics \eqref{MassActionLaw}.

Recall that according to the main result of \cite{FischerReactDiffExistence}, renormalized solutions to entropy-dissipating reaction-diffusion-advection equations exist globally in time under quite general assumptions on the data:
\begin{theorem}[\cite{FischerReactDiffExistence}, Theorem 1]
\label{WeakExistence}
Assume that conditions (A1)-(A6) and (B1) are satisfied; suppose that $g_i$
and $\vec b_i$ meet the conditions in (B2) and (B3). Assume that the reaction rates satisfy the entropy inequality \eqref{EntropyCondition}.
Let $(u_0)_i\in L^1(\Omega)$ be nonnegative functions with
$\sum_{i=1}^S \int_\Omega (u_0)_i \log (u_0)_i ~dx<\infty$.

Then there exists a global (in time) renormalized solution $u$ to
equation (\ref{Equation}) with initial data $u_0$ and boundary conditions (B1)-(B3) in the sense of Definition~\ref{DefinitionSolution}; the solution has the additional regularity $u_i \log u_i \in
L^\infty_{loc}(I;L^1(\Omega))$.
\end{theorem}

Concerning the relation of weak solutions and renormalized solutions, we have shown in \cite{FischerReactDiffExistence} that any renormalized solution $u$ for which the reaction rates $R_i(u)$ are integrable -- that is, for which $R_i(u)\in L^1(\Omega\times [0,T])$ holds for all $i$ -- is also a weak solution. In the present work, we establish a relation in the converse direction:
\begin{theorem}
\label{WeakImpliesRenormalized}
Any weak solution to the reaction-diffusion equation in the sense of Definition \ref{DefinitionWeakSolution} is a renormalized solution in the sense of Definition \ref{DefinitionSolution}.
\end{theorem}
In other words, the notion of weak solutions is a stronger notion of solutions than the notion of renormalized solutions.

Our result on weak-strong uniqueness is based on an adjusted relative entropy inequality for renormalized solutions (see Proposition \ref{RelativeEntropyInequality} below) and the Gronwall lemma. The proof of the adjusted relative entropy inequality makes use of the following entropy dissipation estimate for renormalized solutions, which is also of independent interest as entropy dissipation inequalities may be used to analyze the long-time behavior of solutions (see \cite{DesvillettesFellner2,DesvillettesFellner,DesvillettesFellner4,
DesvillettesFellner3,FellnerTang,MielkeMarkowich} and the references therein):
For a system of reversible chemical reactions with mass-action kinetics subject to the detailed balance condition and the equation \eqref{SimpleReactDiff} with no-flux boundary conditions, the entropy condition holds for the choice $\mu_i:=-\log \lim_{t\rightarrow\infty}\dashint_\Omega u_i(x,t) \,dx$; for this choice, the unique minimum of the entropy $E[u]$ corresponds to the equilibrium state.

\begin{proposition}
\label{EntropyDissipation}
Suppose that the assumptions (A1)-(A6) are satisfied and suppose that the initial data $u_0$ are measurable, nonnegative, and have finite entropy, i.\,e.\ 
\begin{align*}
\sum_{i=1}^S \int_\Omega (u_0)_i (\log (u_0)_i+\mu_i-1) \,dx <\infty.
\end{align*}
Let $u$ be a renormalized solution to the reaction-diffusion-advection equation \eqref{Equation} with initial data $u_0$ and boundary conditions (B1)-(B3) in the sense of Definition~\ref{DefinitionSolution}. Then for almost every $T>0$ the entropy dissipation estimate
\begin{align*}
&\sum_{i=1}^S \int_\Omega u_i (\log u_i+\mu_i-1) \,dx
\bigg|_0^T
\\&
\leq -4\sum_{i=1}^S \int_0^T\int_\Omega A_i\nabla \sqrt{u_i}\cdot \nabla \sqrt{u_i} \,dx\,dt
\\&~~~
+\sum_{i=1}^S \int_0^T \int_\Omega \vec{b_i} \cdot \nabla u_i \,dx\,dt
\\&~~~
+\sum_{i=1}^S \int_0^T \int_{\Gamma_{In}} g_i (\log u_i+\mu_i) \,d{\mathcal{H}}^{d-1} \,dt
\\&~~~
-\sum_{i=1}^S \int_0^T \int_{\Gamma_{Out}} \vec{n}\cdot \vec{b_i} u_i (\log u_i+\mu_i) \,d{\mathcal{H}}^{d-1} \,dt
\\&~~~
+\int_0^T\int_\Omega \sum_{i=1}^S R_i(u)(\log u_i+\mu_i) \,dx\,dt
\end{align*}
is satisfied.
\end{proposition}
In the study of the large-time behavior of reaction-diffusion equations, the conservation laws of the reaction rates $R_i$ typically play a crucial role. For systems of chemical reactions \eqref{SystemOfEquations}, in many cases certain linear combinations of the concentrations $u_i$ are not changed by the reaction, that is for some $q\in \mathbb{R}^S$ the quantity
\begin{align}
\label{ConservedQuantity}
\sum_{i=1}^S q_i \int_\Omega u_i \,dx
\end{align}
may (at least formally) only change by fluxes through the boundary. Translated into a condition on the reaction rates, this is equivalent to the condition
\begin{align}
\label{ConservedVector1}
\sum_{i=1}^S q_i R_i(w)=0
\quad\quad \text{for all }w\in (\mathbb{R}_0^+)^S.
\end{align}
To give an example for such conservation laws, in the case of $N_R<S$ reactions of the form \eqref{SystemOfEquations} there exist at least $S-N_R$ linearly independent vectors $q\in \mathbb{R}^S$ that are orthogonal to all vectors $\beta^n-\alpha^n$, $1\leq n\leq N_R$; hence, there exist $S-N_R$ linearly independent vectors $q$ satisfying \eqref{ConservedVector1}. While it is straightforward to verify that for weak solutions to the reaction-diffusion equation \eqref{Equation} the quantity \eqref{ConservedQuantity} is indeed conserved up to fluxes through the boundary as soon as \eqref{ConservedVector1} holds, for renormalized solutions it is not directly clear whether \eqref{ConservedVector1} entails the conservation of \eqref{ConservedQuantity}. In the following proposition, we prove that the notion of renormalized solutions is indeed compatible with the conservation laws of the reaction rates $R_i$.
\begin{proposition}
\label{ConservationLaws}
Let the assumptions of Proposition~\ref{EntropyDissipation} be satisfied and let $q\in \mathbb{R}^S$ be a vector with the property
\begin{align}
\label{ConservedVector}
\sum_{i=1}^S q_i R_i(w)=0
\quad\quad \text{for any }w\in (\mathbb{R}_0^+)^S.
\end{align}
Then the renormalized solution $u$ is subject to the conservation law
\begin{align*}
&\sum_{i=1}^S q_i \int_\Omega u_i\,dx
\bigg|_0^T
=\int_0^T \int_{\Gamma_{In}} \sum_{i=1}^S q_i g_i \,d\mathcal{H}^{d-1}\,dt
-\int_0^T \int_{\Gamma_{Out}} \sum_{i=1}^S q_i (\vec{n} \cdot \vec{b}_i u_i) \,d\mathcal{H}^{d-1}\,dt
\end{align*}
for almost every $T>0$.
\end{proposition}

Our main result -- the weak-strong uniqueness principle for renormalized solutions to entropy-dissipating reaction-diffusion equations -- reads as follows. Note that in view of Theorem~\ref{WeakImpliesRenormalized} the weak-strong uniqueness result also holds for \emph{weak solutions} $u$ to the reaction-diffusion equation in place of renormalized solutions $u$.
\begin{theorem}[Weak-strong uniqueness of renormalized solutions]
\label{Theorem}
Let the assumptions (A1)-(A6) be satisfied and let the entropy condition \eqref{EntropyCondition} be satisfied by the reaction rates for some $\mu_i\in \mathbb{R}$. Let $u_0\in L^1(\Omega;(\mathbb{R}_0^+)^S)$ be strictly positive initial data subject to the bound
\begin{align*}
\sum_{i=1}^S \int_\Omega (u_0)_i (\log (u_0)_i+\mu_i-1) \,dx <\infty.
\end{align*}
Assume that there exists a ``strong'' solution $v$ to the reaction-diffusion-advection equation \eqref{Equation} with initial data $u_0$ and boundary conditions (B1)-(B3) on some time interval $[0,T_{max})$, that is a weak solution $v$ in the sense of Definition~\ref{DefinitionWeakSolution} with the additional regularities
\begin{align*}
\sup_{x\in \Omega,t\in [0,T_{max})} &v_i(x,t)~~~~~~ < ~\infty,
\\
\inf_{x\in \Omega,t\in [0,T_{max})} &v_i(x,t)~~~~~~ > ~ 0,
\\
\sup_{x\in \Omega,t\in [0,T_{max})} &|\nabla v_i(x,t)| ~~ < ~ \infty,
\\
\sup_{x\in \Omega,t\in [0,T_{max})} &\Big|\frac{d}{dt} v_i(x,t)\Big| ~ < ~ \infty,
\end{align*}
for all $i$.

Let $u$ be any renormalized solution to the reaction-diffusion-advection equation \eqref{Equation} with initial data $u_0$ and boundary conditions (B1)-(B3) in the sense of Definition~\ref{DefinitionSolution}. Then it holds that
\begin{align*}
u(\cdot,T)=v(\cdot,T) \quad\quad\text{almost everywhere in }\Omega
\end{align*}
for almost every $T<T_{max}$.
\end{theorem}

\begin{remark}
Note that our assumption of the existence of a strong solution $v$ in the theorem implicitly puts further regularity constraints on the initial data, the boundary data, and possibly the domain.

On the other hand, by standard regularity theory any bounded weak solution satisfies the regularity assumptions of our theorem, provided that the data and the domain are sufficiently smooth. Furthermore, the required uniform lower bound
\begin{align*}
\inf_{x\in \Omega,t\in [0,T_{max})} &v_i(x,t)~~~~~~ > ~ 0
\end{align*}
is available for any bounded weak solution $v$, provided that the initial data satisfy such a lower bound.
\end{remark}

\section{Proof of the Weak-Strong Uniqueness Principle}

The proof of the weak-strong uniqueness result is structured as follows: In Section~\ref{EntropyInequalitySection}, we establish the entropy estimate for renormalized solutions, that is Proposition~\ref{EntropyDissipation}, as well as the result on conservation laws, that is Proposition~\ref{ConservationLaws}. In Section~\ref{Coercivity}, we show suitable coercivity properties of the adjusted relative entropy functional which will be required for the proof of weak-strong uniqueness. In Section~\ref{RelativeEntropyInequalitySection}, we state and prove Proposition~\ref{RelativeEntropyInequality}, i.\,e.\ the relative entropy inequality satisfied for renormalized solutions $u$ and strong solutions $v$. The proof of Proposition~\ref{RelativeEntropyInequality} makes use of the entropy estimate stated in Proposition~\ref{EntropyDissipation}. In Section~\ref{WeakStrongUniquenessProof}, we demonstrate how the relative entropy inequality for renormalized solutions $u$ and strong solutions $v$ proven in Proposition~\ref{RelativeEntropyInequality} entails the weak-strong uniqueness result.

Note that by $C$ we denote a constant that may change from line to line. However, the values $M$ and $K$ are chosen once and for all, depending on the data and the strong solution $v$.

\subsection{Proof of the entropy dissipation property of renormalized solutions}
\label{EntropyInequalitySection}

\begin{proof}[Proof of Proposition \ref{EntropyDissipation}]
Let $M\geq 2$ and let $\theta_M:\mathbb{R}\rightarrow \mathbb{R}$ be a smooth function with $\theta_M(s)=s$ for $s\leq M$, $0\leq \theta_M'(s)\leq 1$ and $|\theta_M''(s)|\leq \frac{C}{1+s \log (s+1)}$ for all $s\geq 0$, and $\theta_M'(s)=0$ for $s\geq M^C$ for some $C$ large enough. Such a function exists due to $\int_M^{M^C} \frac{1}{1+s \log (s+1)} \,ds \geq 2$ for $C$ large enough.

Choosing the renormalization $\xi(w):=\theta_M\big(\sum_{i=1}^S (w_i+\epsilon) (\log (w_i+\epsilon)+\mu_i-1)\big)$ in \eqref{RenormalizedFormulation} and the test function $\psi\equiv 1$, we obtain
\begin{align*}
&\int_\Omega \theta_M\bigg(\sum_{i=1}^S (u_i+\epsilon) (\log (u_i+\epsilon)+\mu_i-1)\bigg) ~dx \Bigg|_0^T
\\&
=
-\sum_{j=1}^S \int_0^T \int_\Omega \theta_M'\bigg(\sum_{i=1}^S (u_i+\epsilon) (\log (u_i+\epsilon)+\mu_i-1)\bigg)
\frac{(A_j\nabla u_j-u_j \vec b_j) \cdot \nabla u_j}{u_j+\epsilon} ~dx ~dt
\\&~~~
-\sum_{j,k=1}^S \int_0^T \int_\Omega \theta_M''\bigg(\ldots\bigg) (\log (u_j+\epsilon)+\mu_j)(\log (u_k+\epsilon)+\mu_k)
(A_j\nabla u_j-u_j \vec b_j) \cdot \nabla u_k ~dx ~dt
\\&~~~
+\sum_{j=1}^S \int_0^T \int_\Omega \theta_M'\bigg(\sum_{i=1}^S (u_i+\epsilon) (\log (u_i+\epsilon)+\mu_i-1)\bigg) R_j(u) (\log (u_j+\epsilon)+\mu_j) ~dx~dt
\\&~~~
+\sum_{j=1}^S \int_0^T \int_{\Gamma_{In}} \theta_M'\bigg(\sum_{i=1}^S (u_i+\epsilon) (\log (u_i+\epsilon)+\mu_i-1)\bigg) g_j \big(\log (u_j+\epsilon)+\mu_j\big) ~d\mathcal{H}^{d-1} ~dt
\\&~~~
-\sum_{j=1}^S \int_0^T \int_{\Gamma_{Out}}  \theta_M'\bigg(\sum_{i=1}^S (u_i+\epsilon) (\log (u_i+\epsilon)+\mu_i-1)\bigg) (\log(u_j+\epsilon)+\mu_j) \vec n \cdot \vec b_j ~u_j ~d\mathcal{H}^{d-1} ~dt.
\end{align*}
We may then pass to the limit $\epsilon \rightarrow 0$: The passage to the limit in the term on the left-hand side is immediate using Lebesgue's theorem of dominated convergence. Rewriting the gradients $\nabla u_i$ in the form $2\sqrt{u_i}\nabla \sqrt{u_i}$ and taking into account that $\theta'_M(s)$ and $\theta''_M(s)$ vanish for $s\geq M^C$, the passage to the limit in the first two integrals on the right-hand side is also straightforward (using again dominated convergence). In the last integral, that is the boundary integral over $\Gamma_{Out}$, one may also pass to the limit using dominated convergence. For the boundary integral over $\Gamma_{In}$, one may apply Fatou's lemma in connection with the properties $0\leq \theta'_M\leq 1$ and $g_j\geq 0$ (see (B2)): For $\epsilon\leq \frac{1}{2}$ the integrand is bounded from above by $g_j ((\log u_j)_+ + 1+|\mu_j|)$, which is integrable over $\Gamma_{In}\times [0,T]$ by $\sqrt{u_i}\in L^2_{loc}(I;H^1(\Omega))$ and by a boundary trace estimate as well as $g_j \in L^\infty$. Thus, Fatou's lemma is applicable to the negative of the integral and yields
\begin{align*}
&\limsup_{\epsilon\rightarrow 0} \int_0^T \int_{\Gamma_{In}} \theta_M'\bigg(\sum_{i=1}^S (u_i+\epsilon) (\log (u_i+\epsilon)+\mu_i-1)\bigg) g_j \big(\log (u_j+\epsilon)+\mu_j\big) ~d\mathcal{H}^{d-1} ~dt
\\&
\leq \int_0^T \int_{\Gamma_{In}} \theta_M'\bigg(\sum_{i=1}^S u_i (\log u_i+\mu_i-1)\bigg) g_j \big(\log u_j+\mu_j\big) ~d\mathcal{H}^{d-1} ~dt.
\end{align*}
It remains to discuss the third integral on the right-hand side of the above formula, that is the integral containing the reaction terms. In this term, we face the problem that due to the potential failure of strict positivity of $u_j$, we have no guarantee that $\log u_j$ belongs to $L^1(\Omega)$. To deal with this issue, one may use the estimate $R_j(u)\geq -C(M,R_j)u_j$ valid for $\sum_{i=1}^S u_i (\log u_i+\mu_i-1)\leq M^C$ (which is a consequence of the assumption $R_j(w)\geq 0$ in case $w_j=0$ (see (A6)) and the Lipschitz continuity of $R_j$ on bounded subsets of $(\mathbb{R}_0^+)^S$ (see (A5))). Since $\theta'_M(s)$ vanishes for $s\geq M^C$, this bound ensures that $R_j(u)\log (u_j+\epsilon)\leq C(M) u_j |\log (u_j+\epsilon)|$, which is now bounded uniformly in $\epsilon$. Thus, one may apply Fatou's lemma (again with a uniform upper bound instead of a uniform lower bound for the functions, thereby reversing the usual direction of the inequality in Fatou's lemma). In total, we deduce
\begin{align*}
&\int_\Omega \theta_M\bigg(\sum_{i=1}^S u_i (\log u_i+\mu_i-1)\bigg) ~dx \Bigg|_0^T
\\&
\leq
-2\sum_{j=1}^S \int_0^T \int_\Omega \theta_M'\bigg(\sum_{i=1}^S u_i (\log u_i+\mu_i-1)\bigg)
\big(2A_j\nabla \sqrt{u_j}-\sqrt{u_j} \vec b_j\big) \cdot \nabla \sqrt{u_j} ~dx ~dt
\\&~~~
-2\sum_{j,k=1}^S \int_0^T \int_\Omega \theta_M''\bigg(\sum_{i=1}^S u_i (\log u_i+\mu_i-1)\bigg) \sqrt{u_j}\sqrt{u_k} (\log u_j+\mu_j)(\log u_k+\mu_k)
\\&~~~~~~~~~~~~~~~~~~~~~~~~~~\times
\big(2A_j\nabla \sqrt{u_j}-\sqrt{u_j} \vec b_j\big) \cdot \nabla \sqrt{u_k} ~dx ~dt
\\&~~~
+\int_0^T \int_\Omega \theta_M'\bigg(\sum_{i=1}^S u_i (\log u_i+\mu_i-1)\bigg) \sum_{j=1}^S  R_j(u) (\log u_j+\mu_j) ~dx~dt
\\&~~~
+\sum_{j=1}^S \int_0^T \int_{\Gamma_{In}} \theta_M'\bigg(\sum_{i=1}^S u_i (\log u_i+\mu_i-1)\bigg) g_j \big(\log u_j+\mu_j\big) ~d\mathcal{H}^{d-1} ~dt
\\&~~~
-\sum_{j=1}^S \int_0^T \int_{\Gamma_{Out}}  \theta_M'\bigg(\sum_{i=1}^S u_i (\log u_i+\mu_i-1)\bigg) (\log u_j+\mu_j) \vec n \cdot \vec b_j ~u_j ~d\mathcal{H}^{d-1} ~dt.
\end{align*}
We now pass to the limit $M\rightarrow \infty$ to obtain the desired entropy dissipation estimate. By our assumption $(u_0)_i \log (u_0)_i \in L^1(\Omega)$ and the properties of our mappings $\theta_M$, we deduce convergence of the integral on the left-hand side at initial time $t=0$ using the dominated convergence theorem. In the integral at time $t=T$ on the left-hand side, we may pass to the limit by Fatou's lemma. The regularity $\sqrt{u_i}\in L^2_{loc}(I;H^1(\Omega))$ and the properties of our mappings $\theta_M$ enable us to apply the dominated convergence theorem also to the first term on the right-hand side. The reaction term may be dealt with by Fatou's lemma, making use of the entropy dissipation property \eqref{EntropyCondition} and the fact that $\theta_M'\geq 0$. To deal with the last two terms on the right-hand side, we use Fatou's lemma, the fact that $0\leq \theta_M' \leq 1$, $g_j\geq 0$, and $\vec{n}\cdot \vec{b_j}\geq 0$ (recall also the bound $u_i \log u_i \in L^1(\partial\Omega\times [0,T])$, inferred from the bounds $u_i \in L^\infty_{loc}(I;L^1(\Omega))$ and $\sqrt{u_i}\in L^2_{loc}(I;H^1(\Omega))$ by an interpolation-trace estimate). It remains to deal with the second term on the right-hand side. This term, however, is easily seen to vanish in the limit $M\rightarrow \infty$ using dominated convergence and the regularity $\sqrt{u_i}\in L^2_{loc}(I;H^1(\Omega))$: The bound $|\theta_M''(s)|\leq \frac{C}{1+s\log(s+1)}$ implies
\begin{align*}
&\Bigg|\theta_M''\bigg(\sum_{i=1}^S u_i (\log u_i+\mu_i-1) \bigg) \sqrt{u_j}\sqrt{u_k} (\log u_j+\mu_j)(\log u_k+\mu_k)\Bigg|
\\&
\leq
C\frac{|\sqrt{u_j}\sqrt{u_k} (\log u_j+\mu_j)(\log u_k+\mu_k)|}{1+\sum_{i=1}^S u_i (\log (u_i+1))^2}
\leq C,
\end{align*}
which yields the estimate required for the dominated convergence theorem.
\end{proof}

\begin{proof}[Proof of Proposition~\ref{ConservationLaws}]
To prove the conservation law properties, we proceed similarly to the proof of the entropy dissipation estimate in Proposition~\ref{EntropyDissipation}. Let us choose some $\rho \in \mathbb{R}$ and consider the renormalization
\begin{align*}
\xi(w):=\theta_M\left(\rho\sum_{i=1}^S q_i w_i +\sum_{i=1}^S (w_i+\epsilon) (\log (w_i+\epsilon)+\mu_i-1)\right)
\end{align*}
and the test function $\psi\equiv 1$. Using essentially the same arguments as in the proof of Proposition~\ref{EntropyDissipation} -- note that for $M$ large enough (depending on $q$ and $\rho$) the modified $\xi$ satisfies the same estimates as before -- but making additionally use of the cancellation $\sum_{i=1}^S \theta_M'(\ldots) \rho q_i R_i(u) =0$, we deduce by performing the limits $\epsilon\rightarrow 0$ and subsequently $M\rightarrow \infty$ that for almost every $T>0$
\begin{align*}
&\sum_{i=1}^S \int_\Omega \rho q_i w_i + u_i (\log u_i+\mu_i-1) \,dx
\bigg|_0^T
\\&
\leq -4\sum_{i=1}^S \int_0^T\int_\Omega A_i\nabla \sqrt{u_i}\cdot \nabla \sqrt{u_i} \,dx\,dt
\\&~~~
+\sum_{i=1}^S \int_0^T \int_\Omega \vec{b_i} \cdot \nabla u_i \,dx\,dt
\\&~~~
+\sum_{i=1}^S \int_0^T \int_{\Gamma_{In}} g_i \Big(\rho q_i+(\log u_i+\mu_i)\Big) \,d{\mathcal{H}}^{d-1} \,dt
\\&~~~
-\sum_{i=1}^S \int_0^T \int_{\Gamma_{Out}} \vec{n}\cdot \vec{b_i} u_i \Big(\rho q_i+(\log u_i+\mu_i)\Big) \,d{\mathcal{H}}^{d-1} \,dt
\\&~~~
+\int_0^T\int_\Omega \sum_{i=1}^S R_i(u)(\log u_i+\mu_i) \,dx\,dt.
\end{align*}
Dividing both sides by $\rho$ for $\rho>0$ and passing to the limit $\rho\rightarrow \infty$, making use of the entropy dissipation estimate we obtain by dominated convergence
\begin{align*}
\sum_{i=1}^S \int_\Omega q_i w_i \,dx
\bigg|_0^T
\leq
\sum_{i=1}^S \int_0^T \int_{\Gamma_{In}} g_i q_i \,d{\mathcal{H}}^{d-1} \,dt
-\sum_{i=1}^S \int_0^T \int_{\Gamma_{Out}} \vec{n}\cdot \vec{b_i} u_i q_i \,d{\mathcal{H}}^{d-1} \,dt.
\end{align*}
By dividing by $\rho$ for $\rho<0$ and passing to the limit $\rho\rightarrow -\infty$ instead, the reverse inequality is obtained (as the division by the negative number $\rho$ reverses the direction of the inequality). This establishes the conservation property.
\end{proof}

\subsection{Coercivity properties of the adjusted relative entropy functional}
\label{Coercivity}

Let us first collect the properties of the cutoffs $\xi_M:(\mathbb{R}_0^+)^S\rightarrow\mathbb{R}$ that appear in the definition of the adjusted relative entropy
\begin{align*}
E_M[u|v]=\int_\Omega \sum_{i=1}^S \Big( u_i (\log u_i+\mu_i-1) - \xi_M(u) u_i (\log v_i+\mu_i) + v_i \Big) \,dx.
\end{align*}
Besides smoothness, the $\xi_M$ (with $M\geq 2$ and $K\geq 2$) have the properties
\begin{align*}
&0\leq \xi_M(u)\leq 1
~~~~~~~~~~~~~~~~~~~~~~~~~
\text{for all }u,
\\
&\xi_M(u)=1
~~~~~~~~~~~~~~~~~~~~~~~~\,\,\,\,\,\,\,\quad
\text{for all $u$ with }\sum_{i=1}^S u_i\leq M,
\\
&\xi_M(u)=0
~~~~~~~~~~~~~~~~~~~~~~~~\,\,\,\,\,\,\,\quad
\text{for all $u$ with }\sum_{i=1}^S u_i\geq M^K,
\\
&|\partial_j\xi_M(u)|\leq \frac{C}{K \sum_{i=1}^S u_i}
~~~~~~~~~~~~~
\text{for all $u$ and }j,
\\
&|\partial_j\partial_k\xi_M(u)|\leq \frac{C}{K \big|\sum_{i=1}^S u_i\big|^2} ~~~~~~~\text{for all $u$, $j$, and }k.
\end{align*}
Note that one may easily build such $\xi_M$ by setting
\begin{align*}
\xi_M(u):=\theta\left(\frac{\log \sum_{i=1}^S u_i-\log M}{(K-1)\log M}\right)
\end{align*}
where $\theta(s)$ is a usual smooth cutoff with $\theta(s)=1$ for $s\leq 0$, $\theta(s)=0$ for $s\geq 1$, and $0\leq \theta(s)\leq 1$ everywhere.

In the proof of the weak-strong uniqueness principle, we will frequently require the following two coercivity properties of our relative entropy functional.
\begin{lemma}
\label{Lemma7}
For $M$ chosen large enough (depending on $\sup_{i,x,t} v_i(x,t)$, $\mu_i$, and $S$), we have the estimates
\begin{align}
\label{EstimateLargerThanM}
&\int_\Omega \Big(1+\sum_{i=1}^S u_i+\sum_{i=1}^S u_i \log (u_i+1)\Big)\chi_{\{\sum_{i=1}^S u_i\geq M\}}  \,dx
\leq 2E_M[u|v]
\end{align}
and
\begin{align}
\label{EstimateSmallerThanM}
&\int_\Omega \sum_{i=1}^S |u_i-v_i|^2 \chi_{\{\sum_{i=1}^S u_i\leq M\}} \,dx
\leq C(M)E_M[u|v]
\end{align}
with a constant $C(M)$ depending on the $\mu_i$ and on $M$ (and on $S$).
\end{lemma}
\begin{proof}
Both estimates \eqref{EstimateLargerThanM} and \eqref{EstimateSmallerThanM} will be established by purely pointwise estimates for the adjusted relative entropy, distinguishing the cases $\sum_{i=1}^S u_i>M$ and $\sum_{i=1}^S u_i\leq M$.

To establish \eqref{EstimateLargerThanM}, it is sufficient to observe that
\begin{itemize}
\item
for $\sum_{i=1}^S u_i\geq M$ and $M$ large enough (depending on $\sup_{i,x,t} v_i$, $\mu_i$, and $S$) the term $\sum_{i=1}^S u_i \log u_i$ strongly dominates all other terms in the definition of $E_M[u|v]$ and also the term $\sum_{i=1}^S u_i$, and
\item for $\sum_{i=1}^S u_i\leq M$ we have $\xi_M(u)=1$ and the adjusted relative entropy density thus becomes the standard relative entropy density $\sum_{i=1}^S \big(u_i (\log u_i+\mu_i-1)-u_i (\log v_i+\mu_i)+v_i\big)$, which is nonnegative.
\end{itemize}

To prove \eqref{EstimateSmallerThanM}, we use the nonnegativity of the function $\sum_{i=1}^S \big(u_i (\log u_i+\mu_i-1) - \xi_M(u) u_i (\log v_i+\mu_i) + v_i\big)$ on $\{\sum_{i=1}^S u_i>M\}$ (which holds by our assumption of $M$ being large). On the set $\{\sum_{i=1}^S u_i\leq M\}$, we have $\xi_M(u)=1$ and therefore our adjusted relative entropy density reduced to the standard relative entropy density $\sum_{i=1}^S \big(u_i (\log u_i+\mu_i-1)-u_i (\log v_i+\mu_i)+v_i\big)$. We then make use of the uniform convexity of the function $u\mapsto \sum_{i=1}^S u_i (\log u_i+\mu_i-1)$ on bounded subsets of $(\mathbb{R}_0^+)^S$ (note that we may also assume $\sup_{i,x,t} v_i\leq M$ by requiring $M$ to be large enough) to establish the desired lower bound $\frac{1}{C(M)} \sum_{i=1}^S |u_i-v_i|^2$.
\end{proof}

To estimate the boundary terms, we are going to use the following technical lemma.
\begin{lemma}
Let $\Omega$ be a Lipschitz domain and let $u_i$ be nonnegative and satisfy $\sqrt{u_i}\in L^2([0,T];H^1(\Omega))$. Let $v_i$ be Lipschitz and uniformly bounded from below by a positive number. Then for any $\varepsilon>0$ the estimate
\begin{align}
\label{BoundaryTraceEstimate}
&\int_0^T \int_{\partial\Omega} |\sqrt{u_i}-\sqrt{v_i}|^2 \,d\mathcal{H}^{d-1}\,dt
\\&\nonumber
\leq
\varepsilon \int_0^T \int_\Omega u_i\left|\frac{\nabla u_i}{u_i}-\frac{\nabla v_i}{v_i}\right|^2 \chi_{\{\sum_{i=1}^S u_i \leq M\}}\,dx \,dt
\\&~~~\nonumber
+\varepsilon \int_0^T \int_\Omega |\nabla \sqrt{u_i}|^2 \chi_{\{\sum_{i=1}^S u_i > M\}}\,dx \,dt
\\&~~~\nonumber
+C(d,\Omega,\varepsilon,M,\inf_{i,x,t}v_i,\sup_{i,x,t}v_i,\sup_{i,x,t}|\nabla v_i|) \int_0^T  E_M[u|v] \,dt
\end{align}
holds.
\end{lemma}
\begin{proof}
By a standard interpolation-trace inequality, we have for any $\varepsilon>0$
\begin{align*}
&\int_{\partial\Omega} |\sqrt{u_i}-\sqrt{v_i}|^2 \,d\mathcal{H}^{d-1}
\\&
\leq \frac{\varepsilon}{2} \int_\Omega |\nabla (\sqrt{u_i}-\sqrt{v_i})|^2 \,dx
+C(d,\Omega,\varepsilon) \int_\Omega |\sqrt{u_i}-\sqrt{v_i}|^2 \,dx.
\end{align*}
This implies
\begin{align*}
&\int_0^T \int_{\partial\Omega} |\sqrt{u_i}-\sqrt{v_i}|^2 \,d\mathcal{H}^{d-1}\,dt
\\&
\leq \varepsilon \int_0^T \int_\Omega u_i\left|\frac{\nabla u_i}{u_i}-\frac{\nabla v_i}{v_i}\right|^2 \chi_{\{\sum_{i=1}^S u_i \leq M\}}\,dx \,dt
\\&~~~
+\varepsilon \int_0^T \int_\Omega |\sqrt{u_i}-\sqrt{v_i}|^2 \left|\frac{\nabla v_i}{v_i}\right|^2  \chi_{\{\sum_{i=1}^S u_i \leq M\}}\,dx \,dt
\\&~~~
+\varepsilon \int_0^T \int_\Omega |\nabla \sqrt{u_i}|^2 \chi_{\{\sum_{i=1}^S u_i > M\}}\,dx \,dt
\\&~~~
+\varepsilon \int_0^T \int_\Omega |\nabla \sqrt{v_i}|^2 \chi_{\{\sum_{i=1}^S u_i > M\}}\,dx \,dt
\\&~~~
+C(d,\Omega,\varepsilon) \int_0^T \int_\Omega |\sqrt{u_i}-\sqrt{v_i}|^2 \,dx \,dt.
\end{align*}
For $M$ large enough, this implies the desired bound by the coercivity properties of $E_M[u|v]$ established in \eqref{EstimateLargerThanM} and \eqref{EstimateSmallerThanM}.
\end{proof}

\subsection{Estimate for the adjusted relative entropy}
\label{RelativeEntropyInequalitySection}
\begin{proposition}
\label{RelativeEntropyInequality}
Suppose that the assumptions (A1)-(A6) are satisfied and suppose that the initial data $u_0$ are positive, measurable, and have finite entropy, i.\,e.\ 
\begin{align*}
\sum_{i=1}^S \int_\Omega (u_0)_i (\log (u_0)_i+\mu_i-1) \,dx <\infty.
\end{align*}
Let $v$ be a ``strong'' solution to the reaction-diffusion-advection equation \eqref{Equation} with initial data $u_0$ and boundary conditions (B1)-(B3) on some time interval $[0,T_{max})$, that is let $v$ be a weak solution in the sense of Definition~\ref{DefinitionWeakSolution} with the additional regularities
\begin{align*}
\sup_{x\in \Omega,t\in [0,T_{max})} &v_i(x,t)~~~~~~ < ~\infty,
\\
\inf_{x\in \Omega,t\in [0,T_{max})} &v_i(x,t)~~~~~~ > ~ 0,
\\
\sup_{x\in \Omega,t\in [0,T_{max})} &|\nabla v_i(x,t)| ~~ < ~ \infty,
\\
\sup_{x\in \Omega,t\in [0,T_{max})} &\Big|\frac{d}{dt} v_i(x,t)\Big| ~ < ~ \infty,
\end{align*}
for all $i$.

Let $u$ be a renormalized solution to the reaction-diffusion-advection equation \eqref{Equation} with initial data $u_0$ and boundary conditions (B1)-(B3) in the sense of Definition~\ref{DefinitionSolution}. Then for almost every $T\in [0,T_{max})$ the adjusted relative entropy satisfies
\begin{align*}
&E_M[u|v]\bigg|_0^T
:=\int_\Omega \sum_{i=1}^S \Big( u_i (\log u_i+\mu_i-1) - \xi_M(u) u_i (\log v_i+\mu_i) + v_i \Big) \,dx \Bigg|_0^T
\\&
\leq \sum_{i=1}^S \int_0^T \int_\Omega - u_i A_i \frac{\nabla u_i}{u_i} \cdot \frac{\nabla u_i}{u_i} - u_i \xi_M(u) A_i \frac{\nabla v_i}{v_i} \cdot \frac{\nabla v_i}{v_i}
\\&~~~~~~~~~~~~~~~~~~
+u_i \xi_M(u) A_i \frac{\nabla u_i}{u_i}\cdot \frac{\nabla v_i}{v_i}
+u_i \xi_M(u) A_i \frac{\nabla v_i}{v_i}\cdot \frac{\nabla u_i}{u_i} \,dx\,dt
\\&~~~
+\sum_{i,j=1}^S\int_0^T  \int_\Omega \partial_j \xi_M(u) (\log v_i+\mu_i) (A_i\nabla u_i \cdot \nabla u_j + A_j \nabla u_j \cdot \nabla u_i) \,dx\,dt
\\&~~~
+\sum_{i,j,k=1}^S \int_0^T \int_\Omega u_i \partial_j \partial_k \xi_M(u) (\log v_i+\mu_i) A_j\nabla u_j \cdot \nabla u_k \,dx\,dt
\\&~~~
+\sum_{i,j=1}^S \int_0^T \int_\Omega u_i u_j \partial_j \xi_M(u) \bigg(A_j \frac{\nabla u_j}{u_j} \cdot \frac{\nabla v_i}{v_i} + A_i \frac{\nabla v_i}{v_i} \cdot \frac{\nabla u_j}{u_j} \bigg) \,dx\,dt
\\&~~~
+\sum_{i=1}^S \int_0^T \int_\Omega (1-\xi_M(u)) \vec{b_i} \cdot \nabla u_i \,dx\,dt
\\&~~~
-\sum_{i,j=1}^S \int_0^T \int_\Omega \partial_j \xi_M(u) (\log v_i+\mu_i) \big(u_i \vec b_i\cdot \nabla u_j + u_j \vec b_j \cdot \nabla u_i\big) ~dx ~dt
\\&~~~
-\sum_{i,j,k=1}^S \int_0^T \int_\Omega u_i \partial_j\partial_k \xi_M(u) (\log v_i+\mu_i) u_j \vec b_j \cdot \nabla u_k ~dx ~dt
\\&~~~
-\sum_{i,j=1}^S \int_0^T \int_\Omega u_i \partial_j \xi_M(u) u_j \vec{b_j} \cdot \frac{\nabla v_i}{v_i} \,dx\,dt
\\&~~~
-\sum_{i,j=1}^S \int_0^T \int_\Omega u_i \partial_j \xi_M(u) \vec{b_i} \cdot \nabla u_j \,dx\,dt
\\&~~~
+ \int_0^T \int_\Omega \sum_{i=1}^S \bigg(R_i(u)\big(\log u_i+\mu_i-\xi_M(u) (\log v_i+\mu_i)\big) - R_i(v)\Big(\xi_M(u)\frac{u_i}{v_i}-1 \Big)\bigg) \,dx\,dt
\\&~~~
-\sum_{i,j=1}^S \int_0^T \int_\Omega u_i  \partial_j \xi_M(u) R_j(u) (\log v_i+\mu_i)  \,dx\,dt
\\&~~~
+\sum_{i=1}^S \int_0^T \int_{\Gamma_{In}} g_i \big(\log u_i+\mu_i-\xi_M(u)(\log v_i+\mu_i)\big) - g_i \Big(\frac{u_i}{v_i} \xi_M(u)-1\Big) \,d{\mathcal{H}}^{d-1} \,dt
\\&~~~
-\sum_{i,j=1}^S \int_0^T \int_{\Gamma_{In}} g_j (\log v_i+\mu_i) u_i \partial_j \xi_M(u) \,d{\mathcal{H}}^{d-1} \,dt
\\&~~~
-\sum_{i=1}^S \int_0^T \int_{\Gamma_{Out}} \vec{n}\cdot \vec{b_i}\, u_i \big(\log u_i+\mu_i-\xi_M(u)(\log v_i+\mu_i)\big)
\\&~~~~~~~~~~~~~~~~~~~~~~~~~~~~~
-\vec n \cdot  \vec b_i \,v_i \Big(\frac{u_i}{v_i} \xi_M(u)-1\Big)
\,d{\mathcal{H}}^{d-1} \,dt
\\&~~~
+\sum_{i,j=1}^S \int_0^T \int_{\Gamma_{Out}} \vec n \cdot \vec b_j \, u_j (\log v_i+\mu_i) u_i \partial_j \xi_M(u) ~d\mathcal{H}^{d-1} ~dt.
\end{align*}
\end{proposition}

\begin{proof}
%[Proof of Proposition~\ref{RelativeEntropyInequality}]
To show the estimate for the evolution of the adjusted relative entropy $E_M[u|v]$, let us rewrite
\begin{align*}
E_M[u|v]
=&
\int_\Omega \sum_{i=1}^S \Big( u_i (\log u_i+\mu_i-1) - \xi_M(u) u_i (\log v_i+\mu_i) + v_i \Big) \,dx
\\
=&E[u]
-\sum_{i=1}^S \int_\Omega \xi_M(u) u_i (\log v_i+\mu_i) \,dx
+\sum_{i=1}^S \int_\Omega v_i \,dx.
\end{align*}
To estimate the evolution of the entropy $E[u]$, we shall use the entropy dissipation estimate from Proposition~\ref{EntropyDissipation}. Therefore it only remains to evaluate the evolution of the two remaining integrals. Choosing in the definition of renormalized solutions \eqref{RenormalizedFormulation} the renormalization $\xi(w):=w_i \xi_M(w)$ and the test function $\psi:=\log v_i+\mu_i$ (note that $\psi$ is an admissible test function due to our regularity assumptions on $v$; by approximation, one may use arbitrary Lipschitz test functions in \eqref{RenormalizedFormulation}), we infer for almost every $T\in [0,T_{max})$
\begin{align*}
&\int_\Omega \xi_M(u) u_i (\log v_i+\mu_i) ~dx \Bigg|_0^T
\\&
-\int_0^T \int_\Omega \frac{u_i}{v_i} \xi_M(u) \frac{d}{dt}v_i ~dx ~dt
\\&
=
-\sum_{j=1}^S \int_0^T \int_\Omega \partial_j \xi_M(u) (\log v_i+\mu_i)
\big(A_i\nabla u_i-u_i \vec b_i \big) \cdot \nabla u_j ~dx ~dt
\\&~~~
-\sum_{j,k=1}^S \int_0^T \int_\Omega u_i \partial_j\partial_k \xi_M(u) (\log v_i+\mu_i)
\big(A_j\nabla u_j-u_j \vec b_j \big) \cdot \nabla u_k ~dx ~dt
\\&~~~
-\sum_{j=1}^S \int_0^T \int_\Omega \partial_j \xi_M(u) (\log v_i+\mu_i)
\big(A_j\nabla u_j-u_j \vec b_j \big) \cdot \nabla u_i ~dx ~dt
\\&~~~
-\int_0^T \int_\Omega \xi_M(u)
\big(A_i\nabla u_i-u_i\vec{b_i}\big) \cdot \frac{\nabla v_i}{v_i} ~dx ~dt
\\&~~~
-\sum_{j=1}^S \int_0^T \int_\Omega u_i \partial_j \xi_M(u)
\big(A_j\nabla u_j-u_j\vec{b_j}\big) \cdot \frac{\nabla v_i}{v_i} ~dx ~dt
\\&~~~
+\int_0^T \int_\Omega \xi_M(u) R_i(u) (\log v_i+\mu_i) ~dx~dt
\\&~~~
+\sum_{j=1}^S \int_0^T \int_\Omega u_i \partial_j \xi_M(u) R_j(u) (\log v_i+\mu_i) ~dx~dt
\\&~~~
+\int_0^T \int_{\Gamma_{In}} g_i (\log v_i+\mu_i) \xi_M(u)
~d\mathcal{H}^{d-1} ~dt
\\&~~~
+\sum_{j=1}^S \int_0^T \int_{\Gamma_{In}} g_j (\log v_i+\mu_i) u_i \partial_j \xi_M(u)
~d\mathcal{H}^{d-1} ~dt
\\&~~~
-\int_0^T \int_{\Gamma_{Out}} \vec n \cdot \vec b_i ~u_i (\log v_i+\mu_i) \xi_M(u) ~d\mathcal{H}^{d-1} ~dt
\\&~~~
-\sum_{j=1}^S \int_0^T \int_{\Gamma_{Out}} \vec n \cdot \vec b_j ~u_j (\log v_i+\mu_i) u_i \partial_j \xi_M(u) ~d\mathcal{H}^{d-1} ~dt.
\end{align*}
We now integrate by parts with respect to $t$ in the left-hand side of the weak formulation \eqref{WeakFormulation} for our strong solution $v$; note that this is possible by the assumed regularity of $v$. Testing the resulting equation with $\frac{u_i}{v_i} \xi_M(u)-1$ (which is an admissible test function as it belongs to $L^2([0,T_{max});H^1(\Omega))$; by approximation, one sees that one may indeed use arbitrary test functions from the class $L^2([0,T_{max});H^1(\Omega))$ in the equation for $v$), we deduce
\begin{align*}
&\int_0^T \int_\Omega \frac{u_i}{v_i} \xi_M(u) \frac{d}{dt}v_i ~dx ~dt
-\int_\Omega v_i ~dx \Bigg|_0^T
\\&
=
-\int_0^T \int_\Omega \xi_M(u) \frac{(A_i \nabla v_i - v_i \vec{b_i}) \cdot \nabla u_i}{v_i} ~dx ~dt
\\&~~~
-\sum_{j=1}^S \int_0^T \int_\Omega u_i \partial_j \xi_M(u) \frac{(A_i \nabla v_i - v_i\vec{b_i}) \cdot \nabla u_j}{v_i} ~dx ~dt
\\&~~~
+\int_0^T \int_\Omega \xi_M(u) u_i \frac{(A_i \nabla v_i - v_i \vec{b_i}) \cdot \nabla v_i}{|v_i|^2} ~dx ~dt
\\&~~~
+\int_0^T \int_\Omega \Big(\frac{u_i}{v_i} \xi_M(u)-1\Big) R_i(v) ~dx ~dt
\\&~~~
+\int_0^T \int_{\Gamma_{In}} g_i \Big(\frac{u_i}{v_i} \xi_M(u)-1\Big) 
~d\mathcal{H}^{d-1} ~dt
\\&~~~
-\int_0^T \int_{\Gamma_{Out}} \vec n \cdot \vec b_i ~v_i \Big(\frac{u_i}{v_i} \xi_M(u)-1\Big)  ~d\mathcal{H}^{d-1} ~dt.
\end{align*}
Taking the sum with respect to $i$ in the previous two equations and subtracting the resulting equations from the entropy dissipation inequality derived in Proposition~\ref{EntropyDissipation}, we obtain the desired formula by simply reordering terms and noting that the terms of the form
\begin{align*}
\int_0^T \int_\Omega \frac{u_i}{v_i} \xi_M(u) \frac{d}{dt}v_i \,dx\,dt
\end{align*}
and
\begin{align*}
\int_0^T \int_\Omega \xi_M(u) u_i \frac{\vec{b_i} \cdot\nabla v_i}{v_i} \,dx\,dt
\end{align*}
cancel since they both appear once with a positive and once with a negative sign.
\end{proof}

\subsection{Proof of the weak-strong uniqueness property of renormalized solutions}
\label{WeakStrongUniquenessProof}

\begin{proof}[Proof of Theorem~\ref{Theorem}]
We are going to prove the weak-strong uniqueness principle by post-processing the adjusted relative entropy inequality of Proposition~\ref{RelativeEntropyInequality}, thereby reducing the problem to an application of the Gronwall lemma. To this aim, let us rewrite the estimate in Proposition~\ref{RelativeEntropyInequality} as
\begin{align}
\label{RelativeEntropyShort}
E_M[u|v] \Big|_0^T &\leq I+II+III+IV+V+VI+VII+VIII
\\&~~~~
\nonumber
+IX+X+XI+XII+XIII+XIV+XV.
\end{align}
We shall show that for $M$ and $K$ large enough (depending on $\sup_{i,x,t} v_i(x,t)$, $\sup_{i,x,t} \frac{1}{v_i(x,t)}$, $\sup_{i,x,t} |\nabla v_i(x,t)|$, and the data of the problem), an inequality of the form
\begin{align}
\label{RelativeEntropyGronwall}
E_M[u|v]\bigg|_0^T
\leq C(M,K,\mu_i,\inf v_i,\sup v_i,\sup |\nabla v_i|,R_i) \int_0^T E_M[u|v] \,dt
\end{align}
for almost every $T\in [0,T_{max})$ may be derived. By the coinciding initial data of $u$ and $v$ (which entails for the initial relative entropy $E_M[u|v](0)=0$), the Gronwall lemma therefore implies the desired conclusion $u=v$ almost everywhere for almost every $t\leq T_{max}$.

It only remains to establish \eqref{RelativeEntropyGronwall} by estimating the terms on the right-hand side of \eqref{RelativeEntropyShort} one by one. Regarding the dissipation term $I$, we have by (A2) and (A3) (using also that $\xi_M(u)=0$ for $\sum_{i=1}^S u_i\geq M^K$ and that $\xi_M(u)=1$ for $\sum_{i=1}^S u_i\leq M$)
\begin{align*}
I&=\sum_{i=1}^S \int_0^T \int_\Omega - u_i A_i \frac{\nabla u_i}{u_i} \cdot \frac{\nabla u_i}{u_i} - u_i \xi_M(u) A_i \frac{\nabla v_i}{v_i} \cdot \frac{\nabla v_i}{v_i}
\\&~~~~~~~~~~~~~~~~~~~
+u_i \xi_M(u) A_i \frac{\nabla u_i}{u_i}\cdot \frac{\nabla v_i}{v_i}
+u_i \xi_M(u) A_i \frac{\nabla v_i}{v_i}\cdot \frac{\nabla u_i}{u_i} \,dx\,dt
\\&
\leq - \sum_{i=1}^S \int_0^T \int_\Omega u_i A_i \Big(\frac{\nabla u_i}{u_i}-\frac{\nabla v_i}{v_i}\Big)
\cdot \Big(\frac{\nabla u_i}{u_i}-\frac{\nabla v_i}{v_i}\Big) \chi_{\{\sum_{i=1}^S u_i\leq M\}} \,dx\,dt
\\&~~~
-c\sum_{i=1}^S \int_0^T \int_\Omega|\nabla \sqrt{u_i}|^2 \chi_{\{\sum_{i=1}^S u_i>M\}} \,dx\,dt
\\&~~~
+C\sum_{i=1}^S \int_0^T \int_\Omega u_i \frac{|\nabla v_i|^2}{|v_i|^2} \chi_{\{M< \sum_{i=1}^S u_i\leq M^K\}} \,dx\,dt,
\end{align*}
where as usual $c,C>0$ denote constants and where we have used Young's inequality to estimate terms like $u_i \xi_M(u) A_i \frac{\nabla u_i}{u_i}\cdot \frac{\nabla v_i}{v_i}$. This bound entails by \eqref{EstimateLargerThanM}
\begin{align}
I&\leq
\nonumber
-c \sum_{i=1}^S \int_0^T \int_\Omega u_i \left|\frac{\nabla u_i}{u_i}-\frac{\nabla v_i}{v_i}\right|^2 \chi_{\{\sum_{i=1}^S u_i\leq M\}} \,dx\,dt
\\&~~~
\label{Dissipation}
-c\sum_{i=1}^S \int_0^T \int_\Omega|\nabla \sqrt{u_i}|^2 \chi_{\{\sum_{i=1}^S u_i >M \}} \,dx\,dt
\\&~~~
\nonumber
+C\sum_{i=1}^S \sup_{x,t} \frac{|\nabla v_i|^2}{|v_i|^2} \int_0^T E_M[u|v] \,dt.
\end{align}
Before we estimate the remaining terms, let us collect some estimates on the derivatives of $\xi_M$. By the bounds $|\partial_j \xi_M(u)|\leq \frac{C}{K\sum_{i=1}^S u_i}$ and $|\partial_j \partial_k \xi_M(u)|\leq \frac{C}{K|\sum_{i=1}^S u_i|^2}$, we conclude that
\begin{align}
\label{EstimateDxi}
|\partial_j \xi_M(u)\sqrt{u_i}\sqrt{u_j}| &\leq \frac{C}{K},
\\
\label{EstimateD2xi}
|u_i \partial_j \partial_k \xi_M(u) \sqrt{u_j}\sqrt{u_k}| &\leq \frac{C}{K}.
\end{align}
For the term $II$, we have by \eqref{EstimateDxi} and the fact that $\partial_j \xi_M(u)=0$ for $\sum_{i=1}^S u_i\leq M$
\begin{align*}
II&=
\sum_{i,j=1}^S \int_0^T \int_\Omega \partial_j \xi_M(u) (\log v_i+\mu_i) (A_i\nabla u_i \cdot \nabla u_j + A_j \nabla u_j \cdot \nabla u_i) \,dx \,dt
\\&
\leq \frac{C}{K} \sup_{i,x,t} |\log v_i+\mu_i| \sum_{i=1}^S \int_0^T \int_\Omega |\nabla \sqrt{u_i}|^2 \chi_{\{\sum_{i=1}^S u_i >M \}} \,dx \,dt.
\end{align*}
It is now of crucial importance to observe that for $K$ large enough, this upper bound on $II$ may be absorbed in the dissipation term \eqref{Dissipation}.

Concerning the term $III$, we deduce using \eqref{EstimateD2xi} the fact that $\partial_j \partial_k \xi_M(u)=0$ for $\sum_{i=1}^S u_i\leq M$
\begin{align*}
III&=
\sum_{i,j,k=1}^S \int_0^T \int_\Omega u_i \partial_j \partial_k \xi_M(u) (\log v_i+\mu_i) A_j\nabla u_j \cdot \nabla u_k \,dx \,dt
\\
&\leq
\frac{C}{K} \sup_{i,x,t} |\log v_i+\mu_i| \sum_{i=1}^S \int_0^T \int_\Omega |\nabla \sqrt{u_i}|^2 \chi_{\{\sum_{i=1}^S u_i >M \}} \,dx \,dt.
\end{align*}
For $K$ large enough, this upper bound on $III$ may again be absorbed in the dissipation term \eqref{Dissipation}.

Similarly, we have using \eqref{EstimateDxi} and the fact that $\partial_j \xi_M(u)=0$ for $\sum_{i=1}^S u_i \leq M$
\begin{align*}
IV&=
\sum_{i,j=1}^S \int_0^T \int_\Omega u_i u_j \partial_j \xi_M(u) \bigg(A_j \frac{\nabla u_j}{u_j} \cdot \frac{\nabla v_i}{v_i} + A_i \frac{\nabla v_i}{v_i} \cdot \frac{\nabla u_j}{u_j} \bigg) \,dx \,dt
\\
&\leq
\frac{C}{K} \sum_{i=1}^S \int_0^T \int_\Omega |\nabla \sqrt{u_i}|^2 \chi_{\{\sum_{i=1}^S u_i >M \}} \,dx \,dt
\\&~~~
+\frac{C}{K} \sum_{i,j=1}^S \int_0^T \int_\Omega u_j \left|\frac{\nabla v_i}{v_i}\right|^2 \chi_{\{\sum_{i=1}^S u_i >M \}} \,dx \,dt
\\
&\leq
\frac{C}{K} \sum_{i=1}^S \int_0^T \int_\Omega |\nabla \sqrt{u_i}|^2 \chi_{\{\sum_{i=1}^S u_i >M \}} \,dx \,dt
\\&~~~
+\frac{C}{K} \sup_{i,x,t} \left|\frac{\nabla v_i}{v_i}\right|^2 \int_0^T E_M[u|v] \,dt,
\end{align*}
where we have used \eqref{EstimateLargerThanM} in the last step. Note that for $K$ large enough the first term on the right-hand side may be absorbed in \eqref{Dissipation}, while the second term on the right-hand side is of the form of the right-hand side in \eqref{RelativeEntropyGronwall}.

By $\xi_M(u)=1$ for $\sum_{i=1}^S u_i \leq M$, the property $0\leq \xi_M(u)\leq 1$, boundedness of $\vec{b_i}$, Young's inequality, and the estimate \eqref{EstimateLargerThanM}, we deduce
\begin{align*}
V&=
\sum_{i=1}^S \int_0^T \int_\Omega (1-\xi_M(u)) \vec{b_i} \cdot \nabla u_i \,dx\,dt
\\&
\leq \frac{C}{K} \sum_{i=1}^S \int_0^T \int_\Omega |\nabla \sqrt{u_i}|^2  \chi_{\{\sum_{i=1}^S u_i >M \}} \,dx\,dt
\\&~~~
+CK \sum_{i=1}^S \int_0^T \int_\Omega |\sqrt{u_i}|^2  \chi_{\{\sum_{i=1}^S u_i >M \}} \,dx\,dt
\\&
\leq \frac{C}{K} \sum_{i=1}^S \int_0^T \int_\Omega |\nabla \sqrt{u_i}|^2  \chi_{\{\sum_{i=1}^S u_i >M \}} \,dx\,dt
\\&~~~
+CK \int_0^T E_M[u|v] \,dt.
\end{align*}
Again, for $K$ large enough the first term on the right-hand side may be absorbed in \eqref{Dissipation}, while the second term on the right-hand side is of the form of the right-hand side in \eqref{RelativeEntropyGronwall}.

We have by boundedness of $\vec{b_i}$, the estimate \eqref{EstimateDxi}, the fact that $\partial_j \xi_M(u)=0$ for $\sum_{i=1}^S u_i\leq M$, Young's inequality, and the bound \eqref{EstimateLargerThanM}
\begin{align*}
VI&=
-\sum_{i,j=1}^S \int_0^T \int_\Omega \partial_j \xi_M(u) (\log v_i+\mu_i) \big(u_i \vec b_i\cdot \nabla u_j + u_j \vec b_j \cdot \nabla u_i\big) \,dx \,dt
\\&
\leq \frac{C}{K} \sup_{i,x,t} |\log v_i+\mu_i| \sum_{i=1}^S \int_0^T \int_\Omega |\nabla \sqrt{u_i}|^2  \chi_{\{\sum_{i=1}^S u_i >M \}} \,dx\,dt
\\&~~~
+\frac{C}{K}  \sup_{i,x,t} |\log v_i+\mu_i| \int_0^T E_M[u|v] \,dt.
\end{align*}
Again, for $K$ large enough the first term on the r.\,h.\,s.\ may be absorbed in \eqref{Dissipation} and the second term is of the form of the right-hand side in \eqref{RelativeEntropyGronwall}.

Similarly, by boundedness of $\vec{b_i}$, the estimate \eqref{EstimateD2xi}, the fact that $\partial_j \partial_k \xi_M(u)=0$ for $\sum_{i=1}^S u_i\leq M$, the bound \eqref{EstimateLargerThanM}, and Young's inequality, we get
\begin{align*}
VII&=
-\sum_{i,j,k=1}^S \int_0^T \int_\Omega u_i \partial_j\partial_k \xi_M(u) (\log v_i+\mu_i) u_j \vec b_j \cdot \nabla u_k \,dx \,dt
\\&
\leq \frac{C}{K} \sup_{i,x,t} |\log v_i+\mu_i| \sum_{i=1}^S \int_0^T \int_\Omega |\nabla \sqrt{u_i}|^2  \chi_{\{\sum_{i=1}^S u_i >M \}} \,dx\,dt
\\&~~~
+\frac{C}{K}  \sup_{i,x,t} |\log v_i+\mu_i| \int_0^T E_M[u|v] \,dt.
\end{align*}
Once more, for $K$ large enough the first term on the r.\,h.\,s.\ may be absorbed in \eqref{Dissipation} and the second term is of the form of the right-hand side in \eqref{RelativeEntropyGronwall}.

We have by boundedness of $\vec{b_i}$, the properties of $\xi_M$, and \eqref{EstimateLargerThanM}
\begin{align*}
VIII&=
-\sum_{i,j=1}^S \int_0^T \int_\Omega u_i \partial_j \xi_M(u) u_j \vec{b_j} \cdot \frac{\nabla v_i}{v_i} \,dx\,dt
\\&
\leq \frac{C}{K} \sup_{i,x,t} \left|\frac{\nabla v_i}{v_i}\right| \sum_{i=1}^S \int_0^T \int_\Omega u_i \chi_{\{\sum_{i=1}^S u_i >M \}} \,dx\,dt
\\&
\leq \frac{C}{K} \sup_{i,x,t} \left|\frac{\nabla v_i}{v_i}\right| \int_0^T E_M[u|v] \,dt,
\end{align*}
which is of the form of the right-hand side of \eqref{RelativeEntropyGronwall}.

The term $IX$ may be estimated similarly to the term $VI$, resulting in the bound
\begin{align*}
IX&=
-\sum_{i,j=1}^S \int_0^T \int_\Omega u_i \partial_j \xi_M(u) \vec{b_i} \cdot \nabla u_j \,dx\,dt
\\&
\leq \frac{C}{K} \sum_{i=1}^S \int_0^T \int_\Omega |\nabla \sqrt{u_i}|^2  \chi_{\{\sum_{i=1}^S u_i >M \}} \,dx\,dt
\\&~~~
+\frac{C}{K} \int_0^T E_M[u|v] \,dt.
\end{align*}

A crucial part of the argument is the estimate for the contribution of the reaction term.
To estimate
\begin{align*}
X&=
\int_0^T \int_\Omega \sum_{i=1}^S \bigg(R_i(u)\big(\log u_i+\mu_i-\xi_M(u) (\log v_i+\mu_i)\big) - R_i(v)\Big(\xi_M(u)\frac{u_i}{v_i}-1 \Big)\bigg) \,dx\,dt,
\end{align*}
we split the integral into an integral over the set $\{\sum_{i=1}^S u_i\leq M\}$ and an integral over the complement $\{\sum_{i=1}^S u_i>M\}$. On the latter set, we make use of the entropy dissipation property \eqref{EntropyCondition}, that is
\begin{align*}
\sum_{i=1}^S R_i(u)(\log u_i+\mu_i) \leq 0.
\end{align*}
On the set $\{\sum_{i=1}^S u_i\leq M\}$ we have $\xi_M(u)=1$ and thus
\begin{align*}
&R_i(u)\big(\log u_i+\mu_i-\xi_M(u) (\log v_i+\mu_i)\big) - R_i(v)\left(\xi_M(u)\frac{u_i}{v_i}-1 \right)
\\&
=R_i(u) \left(\log \frac{u_i}{v_i}-\frac{u_i}{v_i}+1\right)
+(R_i(u)-R_i(v))\left(\frac{u_i}{v_i}-1\right).
\end{align*}
We therefore have
\begin{align*}
X&\leq \sum_{i=1}^S \int_0^T \int_\Omega R_i(u) \left(\log \frac{u_i}{v_i}-\frac{u_i}{v_i}+1\right) \chi_{\{\sum_{i=1}^S u_i \leq M \}} \,dx\,dt
\\&~~~
+\sum_{i=1}^S \int_0^T \int_\Omega (R_i(u)-R_i(v))\left(\frac{u_i}{v_i}-1\right) \chi_{\{\sum_{i=1}^S u_i \leq M \}} \,dx\,dt
\\&~~~
+\sum_{i=1}^S \int_0^T \int_\Omega \Big(|R_i(u)\xi_M(u)| |\log v_i+\mu_i| + |R_i(v)|\Big(\frac{u_i}{v_i}+1\Big)\Big) \chi_{\{\sum_{i=1}^S u_i > M \}} \,dx\,dt.
\end{align*}
We now use the estimate $\sup_{w\in (\mathbb{R}_0^+)^S} |R_i(w)\xi_M(w)|\leq C(M,K,R_i)$ (which holds due to $\xi_M(u)=0$ for $\sum_{j=1}^S u_j\geq M^K$) as well as the bound $|R_i(u)-R_i(v)|\leq C(M,R_i)\sum_{j=1}^S |u_j-v_j|$ in the case $\sum_{j=1}^S u_j \leq M$ and $\sum_{j=1}^S v_j \leq M$ (which is a consequence of the Lipschitz continuity of $R_i$ on bounded sets, see (A5)). The previous estimate then turns into the bound
\begin{align*}
X&\leq \sum_{i=1}^S \int_0^T \int_\Omega R_i(u) \left(\log \frac{u_i}{v_i}-\frac{u_i}{v_i}+1\right) \chi_{\{\sum_{i=1}^S u_i \leq M \}} \,dx\,dt
\\&~~~
+C(M,R_i) \sup_{i,x,t} \frac{1}{v_i} \int_0^T \int_\Omega \sum_{i=1}^S |u_i-v_i|^2 \chi_{\{\sum_{i=1}^S u_i \leq M \}} \,dx\,dt
\\&~~~
+C(M,K,R_i) \sup_{i,x,t} \left(|\log v_i+\mu_i|+|R_i(v)|+\frac{|R_i(v)|}{v_i} \right)
\\&~~~~~~~~~~~~~~~~~~~~~~~~~
\times
\int_0^T \int_\Omega \Big(1+\sum_{i=1}^S u_i\Big) \chi_{\{\sum_{i=1}^S u_i > M \}} \,dx \,dt.
\end{align*}
Elementary calculus yields the estimate $0\geq \log x-x+1 \geq -|x-1|^2$ for $x\geq 1$ and $0\geq \log x-x+1\geq -\frac{|x-1|^2}{x}$ for $x\leq 1$. By the Lipschitz continuity of $R_i$ on bounded sets (see (A5)) and the fact that $R_i(u)\geq 0$ in case $u_i=0$ (see (A6)), we infer $R_i(u)\geq -C(M,R_i)u_i$ for $\sum_{i=1}^S u_i\leq M$. This implies for $\sum_{i=1}^S u_i\leq M$
\begin{align*}
R_i(u) \left(\log \frac{u_i}{v_i}-\frac{u_i}{v_i}+1\right)
\leq C(M,R_i)u_i \left(1+\frac{v_i}{u_i}\right) \left|\frac{u_i-v_i}{v_i}\right|^2.
\end{align*}
In total, we therefore get using also \eqref{EstimateLargerThanM}
\begin{align*}
X&\leq C(M,R_i) \sup_{i,x,t} \left(\frac{1+v_i}{|v_i|^2}\right) \int_0^T \int_\Omega \sum_{i=1}^S |u_i-v_i|^2 \chi_{\{\sum_{i=1}^S u_i \leq M \}} \,dx\,dt
\\&~~~
+C(M,R_i) \sup_{i,x,t} \frac{1}{v_i} \int_0^T \int_\Omega \sum_{i=1}^S |u_i-v_i|^2 \chi_{\{\sum_{i=1}^S u_i \leq M \}} \,dx\,dt
\\&~~~
+C(M,K,R_i) \sup_{i,x,t} \left(|\log v_i+\mu_i|+|R_i(v)|+\frac{|R_i(v)|}{v_i} \right)
\int_0^T E_M[u|v] \,dt.
\end{align*}
Now note that the first two integrals may be estimated in terms of $C(M)\int_0^T E_M[u|v] \,dt$ due to inequality \eqref{EstimateSmallerThanM}. Therefore, the term $X$ is estimated by a term of the form of the right-hand side of \eqref{RelativeEntropyGronwall}.

To estimate the term $XI$, we have by the bound $|\partial_j \xi_M(u)|\leq \frac{C}{K\sum_{i=1}^S u_i}$, the fact that $\partial_j \xi_M(u)=0$ whenever $\sum_{i=1}^S u_i\leq M$ or $\sum_{i=1}^S u_i\geq M^K$, and \eqref{EstimateLargerThanM}
\begin{align*}
XI&=
-\sum_{i,j=1}^S \int_0^T \int_\Omega \partial_j \xi_M(u) u_i (\log v_i+\mu_i) \, R_j(u) \,dx\,dt
\\&
\leq \frac{C}{K}\sup_{i,x,t} |\log v_i+\mu_i| \sup_{\sum_{j=1}^S w_j\leq M^K} |R(w)| \int_0^T \int_\Omega \chi_{\{\sum_{i=1}^S u_i >M \}} \,dx\,dt
\\&
\leq C(M,K,R_i) \sup_{i,x,t} |\log v_i+\mu_i| \int_0^T E_M[u|v] \,dt.
\end{align*}
It remains to deal with the boundary terms. Rearranging the terms in $XII$, we obtain
\begin{align*}
XII&=
\sum_{i=1}^S \int_0^T \int_{\Gamma_{In}} g_i \big(\log u_i+\mu_i-\xi_M(u)(\log v_i+\mu_i)\big) - g_i \Big(\frac{u_i}{v_i} \xi_M(u)-1\Big) \,d{\mathcal{H}}^{d-1} \,dt
\\&
=\sum_{i=1}^S \int_0^T \int_{\Gamma_{In}} g_i \left(\log \frac{u_i}{v_i}-\frac{u_i}{v_i}+1\right) \,d{\mathcal{H}}^{d-1} \,dt
\\&~~~~
+\sum_{i=1}^S \int_0^T \int_{\Gamma_{In}} g_i (1-\xi_M(u))\left(\log v_i+\mu_i+\frac{u_i}{v_i} \right) \,d{\mathcal{H}}^{d-1} \,dt.
\end{align*}
Making use of the fact that $g_i\geq 0$ (see (B2)) and the fact that $\log x-x+1\leq 0$ for all $x\geq 0$, we see that the first integral on the right-hand side is nonpositive. Estimating the second integral using the boundedness of $g_i$, the fact that $\xi_M(u)=1$ for $\sum_{i=1}^S u_i\leq M$, and the bound $\sum_{i=1}^S v_i \leq \frac{M}{4}$ (which is true for $M$ large enough due to our assumption $\sup_{i,x,t}v_i<\infty$), we infer by $|\sqrt{u_i}-\sqrt{v_i}|^2 \geq \frac{1}{2} u_i - v_i$
\begin{align*}
XII&
\leq
C \sup_{i,x,t} \left(|\log v_i+\mu_i|+\frac{1}{v_i}\right)
\int_0^T \int_{\Gamma_{In}} \Big(1+\sum_{i=1}^S u_i\Big) \chi_{\{\sum_{i=1}^S u_i >M\}} \,d{\mathcal{H}}^{d-1} \,dt
\\&
\leq
C \sup_{i,x,t} \left(|\log v_i+\mu_i|+\frac{1}{v_i}\right)
\int_0^T \int_{\Gamma_{In}} \sum_{i=1}^S |\sqrt{u_i}-\sqrt{v_i}|^2 \,d{\mathcal{H}}^{d-1} \,dt.
\end{align*}
Estimating the boundary integral by the formula \eqref{BoundaryTraceEstimate}, we see that the first two terms resulting from the application of \eqref{BoundaryTraceEstimate} may be absorbed in the dissipation terms of \eqref{Dissipation} by choosing $\varepsilon>0$ small enough. The third resulting term is of the form of the right-hand side in \eqref{RelativeEntropyGronwall}.

Concerning the terms $XIII$ and $XV$, we estimate using the bound $|\partial_j \xi_M(u)|\leq \frac{C}{K\sum_{i=1}^S u_i}$, the fact that $\partial_j \xi_M(u)=0$ for $\sum_{i=1}^S u_i \leq M$, and the boundedness of $\vec{b_j}$ and $g_j$, we infer
\begin{align*}
&XIII+XV
\\
&=
-\sum_{i,j=1}^S \int_0^T \int_{\Gamma_{In}} g_j (\log v_i+\mu_i) u_i \partial_j \xi_M(u) \,d{\mathcal{H}}^{d-1} \,dt
\\&~~~~
+\sum_{i,j=1}^S \int_0^T \int_{\Gamma_{Out}} \vec n \cdot \vec b_j \, u_j (\log v_i+\mu_i) u_i \partial_j \xi_M(u) ~d\mathcal{H}^{d-1} ~dt
\\
&\leq
\frac{C}{K} \sup_{i,x,t} |\log v_i+\mu_i| \int_0^T \int_{\Gamma_{In}} \Big(1+\sum_{i=1}^S  u_i\Big) \chi_{\{\sum_{i=1}^S u_i >M\}} \,d{\mathcal{H}}^{d-1}\,dt
\\
&\leq
\frac{C}{K} \sup_{i,x,t} |\log v_i+\mu_i| \int_0^T \int_{\Gamma_{In}} \sum_{i=1}^S |\sqrt{u_i}-\sqrt{v_i}|^2 \,d{\mathcal{H}}^{d-1}\,dt,
\end{align*}
where in the last step we have assumed that $M$ is so large that $\sum_{i=1}^S v_i \leq \frac{M}{2S}$. As before, we now apply \eqref{BoundaryTraceEstimate} to bound the boundary integral, which after our usual absorption argument leaves only a Gronwall term.

%In case $\sum_{i=1}^S u_i > M$ we have $\vec{n}\cdot b_i u_i (\log u_i+\mu_i)\geq 0$ by our assumption (B3). This entails
Estimating the remaining boundary term $XIV$, we get
\begin{align*}
XIV&=
-\sum_{i=1}^S \int_0^T \int_{\Gamma_{Out}} \vec{n}\cdot \vec{b_i}\, u_i \big(\log u_i+\mu_i-\xi_M(u)(\log v_i+\mu_i)\big)
\\&~~~~~~~~~~~~~~~~~~~~~~~~~~~~~
-\vec n \cdot  \vec b_i \,v_i \Big(\frac{u_i}{v_i} \xi_M(u)-1\Big)
\,d{\mathcal{H}}^{d-1} \,dt
\\&
=
-\sum_{i=1}^S \int_0^T \int_{\Gamma_{Out}} \vec{n}\cdot \vec{b_i} \left(u_i \log \frac{u_i}{v_i} - u_i + v_i\right) \,d{\mathcal{H}}^{d-1} \,dt
\\&~~~~
+\sum_{i=1}^S \int_0^T \int_{\Gamma_{Out}}  \vec{n}\cdot \vec{b_i} (1-\xi_M(u)) u_i (\log v_i+\mu_i+1)
\,d{\mathcal{H}}^{d-1} \,dt
\\&
\leq
C \sup_{i,x,t} |\log v_i+\mu_i+1| \int_0^T \int_{\Gamma_{Out}} \sum_{i=1}^S |\sqrt{u_i}-\sqrt{v_i}|^2
\,d{\mathcal{H}}^{d-1} \,dt.
\end{align*}
Here, in the last step we have used the nonnegativity of the function $a \log a - a \log b -a+b$ for $a,b\geq 0$ as well as the fact that $\vec{n}\cdot \vec{b_i}\geq 0$ on $\Gamma_{Out}$ (see assumption (B3)) to show nonnegativity of the first integral in the estimate; concerning the second integral, we have made use of the bound $M\geq 2S \sum_{i=1}^S v_i$ (which is valid for $M$ large enough) and the fact that $\xi_M(u)=1$ for $\sum_{i=1}^S u_i \leq M$. Again, the remaining boundary integral may be estimated by \eqref{BoundaryTraceEstimate} followed by an absorption argument and leaving behind only a Gronwall term.
\end{proof}

\section{Weak Solutions are Renormalized Solutions}

\begin{proof}[Proof of Theorem \ref{WeakImpliesRenormalized}]
The derivation of the renormalized formulation \eqref{RenormalizedFormulation} of our reaction-diffusion-advection equation from the weak formulation \eqref{WeakFormulation} involves the justification of the chain rule by an appropriate approximation argument.

For compactly supported test functions $\psi\in C^\infty_{cpt}(\Omega \times [0,\infty))$, the approximation argument is relatively straightforward: Let $\rho_\delta$ denote a family of standard symmetric mollifiers (with respect to space, that is $\rho_\delta:\mathbb{R}^d\rightarrow \mathbb{R}$). Given some smooth function $\xi:(\mathbb{R}_0^+)^S\rightarrow \mathbb{R}$ with compactly supported derivative $D\xi$, we consider the test functions
\begin{align*}
\psi_i:=\rho_\delta\ast (\psi \partial_i \xi(\rho_\delta \ast u)),
\end{align*}
where the convolutions all refer to space only. Note that for $\delta>0$ small enough, the function $\psi_i$ is well-defined and belongs to $C^\infty_{cpt}(\Omega\times [0,\infty))$. From the weak formulation \eqref{WeakFormulation} we infer the property $u\in W^{1,1}([0,T];(W^{1,\infty}(\Omega))')$ for all $T\geq 0$, which implies $\rho_\delta\ast u\in W^{1,1}([0,T];C^2(\Omega_\delta))$ with $\Omega_\delta:=\{x\in \Omega:\operatorname{dist}(x,\partial\Omega)>\delta\}$ (and therefore the same regularity for $\psi_i$ in $\Omega_{2\delta}$). Testing the weak formulation \eqref{WeakFormulation} with $\psi_i$, noting that the boundary terms vanish due to the compact support of $\psi_i$, and using general properties of convolutions as well as the symmetry of $\rho_\delta$, we deduce
\begin{align*}
&\int_\Omega (\rho_\delta \ast u_i(\cdot,T)) (\psi \partial_i \xi(\rho_\delta \ast u))(\cdot,T) ~dx
-\int_\Omega (\rho_\delta \ast (u_0)_i) \psi(\cdot,0) \partial_i \xi(\rho_\delta \ast u_0) ~dx
\\&
-\int_0^T \int_\Omega (\rho_\delta \ast u_i) \frac{d}{dt} (\psi \partial_i \xi(\rho_\delta \ast u)) ~dx ~dt
\\
=&-\int_0^T \int_\Omega (\rho_\delta \ast (A_i\nabla u_i)) \cdot \nabla (\psi \partial_i \xi(\rho_\delta \ast u)) ~dx ~dt
\\&
+\int_0^T \int_\Omega (\rho_\delta \ast (u_i \vec b_i)) \cdot \nabla (\psi \partial_i \xi(\rho_\delta \ast u)) ~dx ~dt
\\&
+\int_0^T \int_\Omega (\rho_\delta\ast R_i(u)) \psi \partial_i \xi(\rho_\delta \ast u) ~dx~dt.
\end{align*}
Integrating by parts with respect to time in the time integral on the left-hand side (note that our differentiability properties are now sufficient to justify that), taking the sum with respect to $i$, and employing the chain rule and product rule in several terms (note that again, the regularity is sufficient to justify these), we get
\begin{align*}
&\int_\Omega \xi(\rho_\delta \ast u) \psi ~dx \bigg|_0^T
-\int_0^T \int_\Omega \xi(\rho_\delta \ast u) \frac{d}{dt}\psi ~dx ~dt
\\&
=\int_0^T \int_\Omega \psi \frac{d}{dt} \xi(\rho_\delta \ast u) ~dx ~dt
\\&
=\int_0^T \int_\Omega \sum_{i=1}^S \psi \partial_i \xi(\rho_\delta \ast u) \frac{d}{dt} (\rho_\delta \ast u_i) ~dx ~dt
\\&
=-\sum_{i,j=1}^S \int_0^T \int_\Omega \psi \partial_i \partial_j \xi(\rho_\delta \ast u) (\rho_\delta \ast (A_i\nabla u_i)) \cdot \nabla (\rho_\delta \ast u_j) ~dx ~dt
\\&~~~
-\sum_{i=1}^S \int_0^T \int_\Omega \partial_i \xi(\rho_\delta \ast u) (\rho_\delta \ast (A_i\nabla u_i)) \cdot \nabla \psi ~dx ~dt
\\&~~~
+\sum_{i,j=1}^S \int_0^T \int_\Omega \psi \partial_i \partial_j \xi(\rho_\delta \ast u) (\rho_\delta \ast (u_i \vec b_i)) \cdot \nabla (\rho_\delta \ast u_j) ~dx ~dt
\\&~~~
+\sum_{i=1}^S \int_0^T \int_\Omega \partial_i \xi(\rho_\delta \ast u) (\rho_\delta \ast (u_i \vec b_i)) \cdot \nabla \psi ~dx ~dt
\\&~~~
+\int_0^T \int_\Omega (\rho_\delta\ast R_i(u)) \psi \partial_i \xi(\rho_\delta \ast u) ~dx~dt.
\end{align*}
In order to justify the renormalized formulation \eqref{RenormalizedFormulation} for such compactly supported test functions $\psi$, it only remains to pass to the limit $\delta\rightarrow 0$ in all terms on the right-hand side and the terms in the first line. For example, concerning the last term, we have the convergence $\rho_\delta\ast R_i(u)\rightarrow R_i(u)$ strongly in $L^1(\Omega\times [0,T])$ (as we have by the definition of weak solutions $R_i(u)\in L^1(\Omega \times [0,T])$) and the convergence $\partial_i \xi(\rho_\delta \ast u)\rightarrow \partial_i \xi(u)$ pointwise a.\,e.\ with a uniform $L^\infty$ bound, which together by Vitali's theorem is sufficient for the passage to the limit. For the passage to the limit in the second term on the right-hand side, we also use Vitali's theorem, now with the convergence $\rho_\delta \ast (A_i \nabla u_i)\rightarrow A_i \nabla u_i$ strongly in $L^1(\Omega\times [0,T])$ which holds due to $\sqrt{u_i}\in L^2([0,T];H^1(\Omega))$. The passage to the limit in the terms on the left-hand side and in the fourth term on the right-hand side is similarly accomplished. In the first and the third term on the right-hand side, one additionally needs estimates like $|\rho_\delta \ast (\sqrt{u_i}A_i \nabla \sqrt{u_i})|(x)\leq \sqrt{(\rho_\delta \ast u_i)(x)(\rho_\delta \ast  |A_i \nabla \sqrt{u_i}|^2)(x)}$ (which are a consequence of H\"older's inequality) to ensure appropriate equi-integrability of the integrands (note that the factor $\partial_i \partial_j \xi(\rho_\delta \ast u)$ vanishes whenever one of the $u_j$ becomes too large, thereby eliminating factors like $\sqrt{(\rho_\delta \ast u_i)(x)}$ from the integrability considerations).

In the case of test functions $\psi=\psi(x,t)$ that are nonzero also at the (spatial) boundary $\partial \Omega$, one first observes that (by a decomposition of unity argument) it is sufficient to justify the renormalized formulation for test functions supported in small coordinate patches around boundary points. One then performs a (bi-Lipschitz and volume-preserving) change of variables to straighten the boundary locally, noting that the structure of the equation \eqref{Equation} is stable under such a change of coordinates (while the $A_i$, $\vec{b_i}$ may change, their bounds are preserved up to constant factors). In the changed coordinates, one extends the solution to the other side of the boundary by reflection and proceeds by a similar mollification argument. As the test functions are now nonzero on (a small part of) the boundary, additionally the (desired) boundary terms in \eqref{RenormalizedFormulation} appear. For details of such an argument, see \cite[Proof of Lemma 4]{FischerReactDiffExistence}, where this argument is carried out in detail in another situation.
\end{proof}

\bibliographystyle{abbrv}
\bibliography{reactdiff}

\end{document}